\documentclass{article}

\usepackage{cancel}\usepackage{amsfonts,amssymb}
\usepackage[latin1]{inputenc}
 \usepackage[T1]{fontenc}
\usepackage{tikz}
\usetikzlibrary{matrix}

\usepackage{amsmath}

\newcommand\cone{{\mathcal C}}
\newcommand\ccup{\sqcup}
\newcommand\bo{<}
\newcommand\wbo{\lessdot}
\newcommand\inv{^{-1}}
\newcommand\KoL{{\mathcal L}}\newcommand\KoE{{\mathcal E}}\newcommand\KoR{{\mathcal R}}
\newcommand\End{\operatorname{End}}\newcommand\Ho{\operatorname{H}}

\newcommand\lleq{{\preccurlyeq}}\newcommand\ggeq{{\succcurlyeq}}\newcommand\sggeq{{\succ}}\newcommand\slleq{{\prec}}
\newcommand\nggeq{{\not\hspace{-0.5ex}\ggeq}}\newcommand\nlleq{{\not\hspace{-0.5ex}\lleq}}
\newcommand\G{\operatorname{Gr}}
\newcommand\AAf{{\Omega}}

\newcommand\SO{\operatorname{SO}}\newcommand\Sp{\operatorname{Sp}}\newcommand\SL{\operatorname{SL}}\newcommand\GL{\operatorname{GL}}

\newcommand\kbprod{{\odot_0}}\newcommand\bkprod{\kbprod}

\newcommand\Ker{\operatorname{Ker}}\newcommand\tr{\operatorname{tr}}\renewcommand\Im{\operatorname{Im}}
\newcommand\Ad{\operatorname{Ad}}

\newcommand\Hom{\operatorname{Hom}}
\newcommand\rk{\operatorname{rk}}

\newcommand\Greg{^{X(Z)-{\rm reg}}}

\newcommand\Id{\operatorname{Id}}
\newcommand\lieg{{\mathfrak g}}
\newcommand\liel{{\mathfrak l}}
\newcommand\lk{{\mathfrak k}}
\newcommand\lp{{\mathfrak p}}\newcommand\lu{{\mathfrak u}}\newcommand\lr{{\mathfrak r}}

\newcommand\CC{{\mathbb C}}\newcommand\PP{{\mathbb P}}\newcommand\ZZ{{\mathbb
    Z}}\newcommand\NN{{\mathbb N}}\newcommand\RR{{\mathbb R}}
\newcommand\QQ{{\mathbb Q}}

\newcommand\longto{\longrightarrow}

\newcommand\Sym{{\mathcal S}}

\newcommand\Li{{\mathcal L}}

\newtheorem{lemma}{Lemma}
\newtheorem{conj}{Conjecture}
\newtheorem{prop}{Proposition}
\newtheorem{theo}{Theorem}
\newtheorem{coro}{Corollary}

\newenvironment{proof}{{\noindent\bf Proof.}}{\hfill $\square$}
\newenvironment{defin}{{\noindent\bf Definition.}}{}
\newenvironment{remark}{{\noindent\bf Remark.}}{}
\newenvironment{exple}{{\noindent\bf Example.}}{}

\newcounter{paragrafsubsub}[subsubsection]
\renewcommand{\theparagrafsubsub}{%
\thesubsubsection.\roman{paragrafsubsub}}
\newcommand{\paragrafsubsub}{%
\refstepcounter{paragrafsubsub}
{\bf \theparagrafsubsub}\hspace{0.2em}--- }

\newcounter{paragrafsub}[subsection]
\renewcommand{\theparagrafsub}{\thesubsection.\arabic{paragrafsub}}
\newcommand{\paragrafsub}{%
\refstepcounter{paragrafsub}
{\bf \theparagrafsub}\hspace{0.2em}--- }

\newcounter{paragraf}[section]
\renewcommand{\theparagraf}{\thesection.\arabic{paragraf}}
\newcommand{\paragraf}{%
\refstepcounter{paragraf}
{\bf \theparagraf}\hspace{0.2em}--- }

\newcommand\paragraphe{%
\par \indent
\ifcase\value{subsection} %
\paragraf
\else
\ifcase\value{subsubsection}\paragrafsub %
\else\paragrafsubsub
\fi\fi
}

\begin{document}
\title{Distributions on homogeneous spaces and applications}

\author{N. Ressayre\footnote{Universit{\'e} Montpellier II - 
CC 51-Place Eug{\`e}ne Bataillon -
34095 Montpellier Cedex 5 -
France - {\tt ressayre@math.univ-montp2.fr}}}

\maketitle

\begin{abstract}
Let $G$ be a complex semisimple algebraic group. 
In 2006, Belkale-Kumar defined a new product $\bkprod$ on the
cohomology group $\Ho^*(G/P,\CC)$ of any projective $G$-homogeneous
space $G/P$.
Their definition uses the notion of Levi-movability for triples of
Schubert varieties in $G/P$.

In this article, we introduce a family of $G$-equivariant subbundles
of the tangent bundle of $G/P$ and the associated filtration of the De
Rham complex of $G/P$ viewed as a manifold. 
As a consequence one gets a filtration of the ring $\Ho^*(G/P,\CC)$
and prove that $\bkprod$ is the associated graded product.
One of the aim of this more intrinsic construction of $\bkprod$ is
that there is a natural notion of fundamental class
$[Y]_\bkprod\in(\Ho^*(G/P,\CC),\bkprod)$ 
for any irreducible subvariety $Y$ of $G/P$.

Given two Schubert classes $\sigma_u$ and $\sigma_v$ in
$\Ho^*(G/P,\CC)$, we define a subvariety $\Sigma_u^v$ of $G/P$. 
This variety should play the role of the Richardson variety; more
precisely, we conjecture that
$[\Sigma_u^v]_\bkprod=\sigma_u\bkprod\sigma_v$.
We give some evidence for this conjecture,  and prove special cases.

Finally, we use the subbundles of $TG/P$ to give a geometric
characterization of the $G$-homogeneous locus of any Schubert
subvariety of $G/P$.
\end{abstract}

\section{Introduction}

Let $G$ be a complex semisimple group and let $P$ be a parabolic subgroup
of $G$. 
In this paper, we are interested in the Belkale-Kumar product $\bkprod$
on the cohomology group of the flag variety $G/P$.

\bigskip
\noindent
{\bf The Belkale-Kumar product.}
Fix a maximal torus $T$ and a Borel subgroup $B$ such that
$T\subset B\subset P$. Let $W$ and $W_P$ denote respectively the Weyl groups of $G$
and $P$.
Let $W^P$ be the set of  minimal length representative in the cosets of
$W/W_P$.
For any $w\in W^P$, let $X_w$ be the corresponding Schubert variety 
(that is, the closure of $BwP/P$) and let $[X_w]\in \Ho^*(G/P,\CC)$ be its 
cohomology class.
  The structure coefficients $c_{uv}^w$ of the cup product are written as
  \begin{eqnarray}
    \label{eq:defc}
    [X_u].[X_v]=\sum_{w\in W^P}c_{uv}^w[X_w].
  \end{eqnarray}
Let $L$ be the Levi subgroup of $P$ containing $T$. This group acts on the
tangent space $T_{P/P}G/P$ of $G/P$ at the  base point $P/P$.
Moreover, this action is multiplicity free and we have a unique
decomposition 
\begin{eqnarray}
  \label{eq:decgp}
  T_{P/P}G/P=V_1\oplus\cdots\oplus V_s,
\end{eqnarray}
as sum of irreducible $L$-modules. 
It turns out that, for any $w\in W^P$, the tangent space 
$T_w:=T_{P/P}w^{-1}X_w$ of the variety $w^{-1}X_w$
at the smooth point $P/P$ decomposes as 
\begin{eqnarray}
  \label{eq:decTw}
  T_w=(V_1\cap T_w)\oplus\cdots\oplus (V_s\cap T_w).
\end{eqnarray}
Set $T^i_w:=T_w\cap V_i$.
Since $[X_w]$ has degree $2(\dim(G/P)-\dim(T_w))$ in the graded algebra
$\Ho^*(G/P)$, if $c_{uv}^w\neq 0$ then
\begin{eqnarray}
  \label{eq:conddim}
  \dim(T_u)+\dim(T_v)=\dim(G/P)+\dim(T_w),
\end{eqnarray}
or equivalently
\begin{eqnarray}
  \label{eq:conddim2}
  \sum_{i=1}^s\bigg (\dim(T_u^i)+\dim(T_v^i)\bigg)=
\sum_{i=1}^s\bigg(\dim(V_i)+\dim(T_w^i)\bigg).
\end{eqnarray}
The Belkale-Kumar product requires  the  equality~\eqref{eq:conddim2} to
hold term by term.
More precisely, the structure constants $\tilde c_{uv}^w$ of the Belkale-Kumar product \cite{BK},
\begin{eqnarray}
    \label{eq:defctilde1}
    [X_u]\bkprod [X_v]=\sum_{w\in W^P}\tilde c_{uv}^w[X_w]
  \end{eqnarray}
can be defined as follows (see \cite[Proposition~2.4]{RR}):
\begin{eqnarray}
  \label{eq:defctilde2}
  \tilde c_{uv}^w=\left\{
    \begin{array}{ll}
      c_{uv}^w &\mbox{ if } \forall 1\leq i\leq s\ \ \dim(T^i_u)+\dim(T^i_v)=
\dim(V_i)+\dim(T^i_w),\\
0 &\mbox{ otherwise.}
    \end{array}
\right .
\end{eqnarray}
The product $\bkprod$ defined in such a way is  associative and satisfies Poincar\'e duality. 
The Belkale-Kumar product was proved to be the
more relevant product for describing the Littlewood-Richardson cone 
(see \cite{BK,GITEigen,GITEigen2}).

\bigskip
\noindent
{\bf Motivations.}
If $G/P$ is cominuscule then $T_{P/P}G/P$ is an irreducible $L$-module 
(that is, $s=1$).
In this case, the Belkale-Kumar product is simply the cup product.
This paper is motivated by the guess that several known results for 
cominuscule $G/P$ could be generalized to any $G/P$  using
the Belkale-Kumar product.
In particular, it might be a first step toward a positive geometric
uniform combinatorial rule for computing the coefficients $\tilde c_{uv}^w$.
Indeed, we define a subvariety $\Sigma_u^v$ which is encoded by combinatorial datum
(precisely a subset of roots of $G$).
We also define a Belkale-Kumar fundamental class $[\Sigma_u^v]_{\bkprod}$ and 
conjecture that $[\Sigma_u^v]_{\bkprod}=[X_u]\bkprod [X_v]$. 

\bigskip
\noindent
{\bf A geometric construction of the Belkale-Kumar ring.}
The first aim of this paper is to give a geometric construction of the 
Belkale-Kumar ring which does not deal  with the Schubert basis.
Consider the connected center $Z$ of $L$ and its character group $X(Z)$. 
The
Azad-Barry-Seitz theorem (see \cite{AzBaSe}) asserts that each $V_i$ in 
the  decomposition~\eqref{eq:decgp} is an isotipical component for the action 
of $Z$ associated to some  weight denoted by $\alpha_i\in X(Z)$.
The group $P$ acts on $T_{P/P}G/P$ but does not stabilize the
decomposition~\eqref{eq:decgp}. But, the group $X(Z)$ is endowed with 
 a partial order $\ggeq$ (see Section~\ref{sec:GP} for
 details), such that for any $\alpha\in X(Z)$
the sum
\begin{eqnarray}
  \label{eq:sum}
  V^{\ggeq\alpha}:=\oplus_{\alpha_i\ggeq\alpha}V_i
\end{eqnarray}
is $P$-stable. 
Since $V^{\ggeq\alpha}$ is $P$-stable, it induces a $G$-homogeneous subbundle
$T^{\ggeq\alpha}G/P$ of the tangent bundle $TG/P$. 
We obtain a family of distributions indexed by $X(Z)$.
This family  forms a decreasing multi-filtration: if $\alpha\ggeq\beta$ then
$T^{\ggeq\alpha}G/P$ is a subbundle of $T^{\ggeq\beta}G/P$. Moreover, 
these distributions are globally integrable in the sense that 
\begin{eqnarray}
  \label{eq:integ}
  [T^{\ggeq\alpha}G/P,T^{\ggeq\beta}G/P]\subset T^{\ggeq\alpha+\beta}G/P.
\end{eqnarray}
This allows us to define a filtration (``à la Hodge'') of the De Rham complex 
and so of the algebra $\Ho^*(G/P,\CC)$ indexed by the group $X(Z)\times\ZZ$.
We consider the associated graded algebra.

\begin{theo}
  \label{th:BKprodfil}
 The $(X(Z)\times\ZZ)$-graded algebra $\G \Ho^*(G/P,\CC)$ associated to the 
 $(X(Z)\times\ZZ)$-filtration is  isomorphic to the Belkale-Kumar
  algebra $(\Ho^*(G/P,\CC),\bkprod)$.
\end{theo}

 The first step to get Theorem~\ref{th:BKprodfil} is to give it a precise sense
defining the orders on $X(Z)$ and $X(Z)\times\ZZ$ and the filtrations.
The key point to prove the isomorphism is that the Schubert basis $([X_w])_{w\in W^P}$ of
$\Ho^*(G/P,\CC)$ is adapted to the filtration. Indeed each subspace of the filtration
is spanned by the Schubert classes it contains.
To obtain this result, we use  Kostant's harmonic forms
\cite{Kostant:harmform1}. 
 
Theorem~\ref{th:BKprodfil} is closed to \cite[Theorem~43]{BK} obtained
by Belkale-Kumar.
In \cite{BK}, the filtration on $\Ho^*(G/P,\CC)$ is defined using the
Schubert basis. On the other hand, the filtration on
$\Ho_{DR}^*(G/P,\CC)$ is defined using Kostant's $K$-invariant forms
(where $K$ is a compact form of $G$). 
Here, the filtration is defined independently of any basis or any
choice of a compact form of $G$.

This ``intrinsic'' definition of the Belkale-Kumar also gives a
pleasant interpretation of the functoriality result of
\cite[Theorem~1.1]{RR}. Indeed, let $\tau$ be a one-parameter subgroup of $Z$
such that 
$$
\forall \alpha\in X(Z)\qquad \alpha\ggeq 0\,\Rightarrow\,\langle
\tau,\alpha\rangle\geq 0,
$$
and 
\begin{equation}
  \label{eq:taureg}
\forall 1\leq i\neq j\leq s\qquad \langle
\tau,\alpha_i\rangle\neq \langle
\tau,\alpha_j\rangle.
\end{equation}
Setting for any $n\in\ZZ$
$$
V^{\geq n}:=\oplus_{\langle
\tau,\alpha\rangle\geq n}V_i,
$$
one gets a globally integrable family of distributions on $G/P$
indexed by $\ZZ$.  
Then, one gets a $\ZZ$-filtration of the ring $\Ho^*(G/P,\CC)$. 
By \eqref{eq:defctilde1} and \eqref{eq:taureg}, the associated
$\ZZ$-graded  ring is isomorphic to $\G \Ho^*(G/P,\CC)$. 
Then, \cite[Theorem~1.1]{RR} is a direct consequence of the 
immediate lemma~\ref{lem:funct} below.

\bigskip
\noindent
{\bf A conjecture.}
The main motivation to show Theorem~\ref{th:BKprodfil} is to define the fundamental
class for the Belkale-Kumar product of any irreducible subvariety $Y$ 
of $G/P$.
This class $[Y]_{\bkprod}$ which belongs to $\G \Ho^*(G/P,\CC)$ is defined
 in Section~\ref{sec:defclassY}.

Let $w_0$ and $w_0^P$ be the longest elements of $W$ and $W_P$
respectively. 
If $v\in W^P$ then  $v^\vee:=w_0vw_0^P$ belongs to $W^P$ and $[X_{v^\vee}]$ is 
the Poincaré dual class of $[X_v]$.
Consider the weak Bruhat order $\wbo$ on $W^P$.
We are interested in the product
$[X_u]\bkprod [X_v]\in \Ho^*(G/P,\CC)$, for given $u$ and $v$ in $W^P$. 
Lemma~\ref{lem:BKprodwbo} below shows that if  $[X_u]\bkprod [X_v]\neq 0$ then $v^\vee\wbo
u$.
Assume that  $v^\vee\wbo u$ and consider the group
\begin{eqnarray}
  \label{eq:839}
  H_u^v:=u^{-1}Bu\,\cap\,w_0^Pv^{-1}Bvw_0^P.
\end{eqnarray}
It is a closed connected subgroup of $G$ containing $T$; in
particular, it can be encoded by its set $\Phi_u^v$ of roots.
Let $\Sigma_u^v$ denote the closure of the $H_u^v$-orbit of $P/P$:
\begin{eqnarray}
  \label{eq:840}
  \Sigma_u^v=\overline{H_u^v.P/P}.
\end{eqnarray}

Another characterization of this subvariety is given by the following
statement.

\begin{prop}
\label{prop:defYuv}
  The variety $\Sigma_u^v$ is the unique irreducible component of the intersection
$u^{-1}X_u\cap w_0^Pv^{-1}X_v$ containing $P/P$. 
Moreover, this intersection is transverse along $\Sigma_u^v$.
\end{prop}

Our main conjecture can be stated as follow.

\begin{conj}
  \label{conj1}
If  $v^\vee\wbo u$ then
$$
[\Sigma_u^v]_\bkprod=[X_u]\bkprod[X_v]
\in \G \Ho^*(G/P,\CC).
$$
\end{conj}

Write
$$
[\Sigma_u^v]_\bkprod=\sum_{w\in W^P}d_{uv}^w[X_w].
$$  
By Proposition~\ref{prop:BKclass} $d_{uv}^w$ are integers. Moreover,
Conjecture~\ref{conj1} is equivalent to $\tilde c_{uv}^w=d_{uv}^w$ for any $w\in W^P$.
The first evidence is the following weaker result.

\begin{prop}
\label{prop:support}

Then
\begin{enumerate}
\item $d_{uv}^w\neq 0\iff\tilde c_{uv}^w\neq 0$;
\item $d_{uv}^w\leq\tilde c_{uv}^w$.
\end{enumerate}
\end{prop}

\bigskip
\noindent
{\bf Known cases.}
Conjecture~\ref{conj1} generalizes another one for $G/B$. Indeed, 
if $G/P=G/B$ is a complete
flag variety then Conjecture~\ref{conj1} is equivalent to the following
one.

\begin{conj}
  \label{conjGBi}
For $G/B$ and any $u,v$, and $w$ in $W$, 
the structure constant $\tilde c_{uv}^w$ is equal to $1$ if for any
$1\leq i\leq s$,  $\dim(T^i_u)+\dim(T^i_v)=
\dim(V_i)+\dim(T^i_w)$ and $0$ otherwise.
\end{conj}

In particular, Conjecture~\ref{conjGBi} implies that we have a
uniform combinatorial and geometric model for the Belkale-Kumar product.
Conjecture~\ref{conjGBi} was explicitly stated in  \cite{DR:prv1}.
E.~Richmond proved in \cite{Richmond:recursion} and \cite{Rich:mult} this conjecture for $G=\SL_n (\CC)$
or $G=\Sp_{2n}(\CC)$. In Section~\ref{sec:GB}, we  prove it for
$G=\SO_{2n+1}(\CC)$ (this proof is certainly known from some specialists but
I have shortly included it for convenience). Very rencently, Dimitrov-Roth got also a
proof for classical groups and G2 \cite{DR}.
Using \cite[Corollary~44]{BK}, we wrote a program \cite{MaPage} to check this
conjecture: it is checked in type $F_4$ and $E_6$.
\\

Conjecture~\ref{conj1} will be proved in type A in a forthcoming paper.

\bigskip
\noindent
{\bf Combinatorial evidences.}
Consider the following degenerate version of
Conjecture~\ref{conj1}.

\begin{conj}
 \label{conjcomb}
The product $[X_u]\bkprod[X_v]$ only depends on the set $\Phi_u^v$.
\end{conj}

The expression of the Belkale-Kumar structure coefficients as products
given in \cite{Richmond:recursion} shows that Conjecture~\ref{conjcomb}
holds in type A. 
Consider now the case $G=\SO_{2n+1}(\CC)$ or $\Sp_{2n}(\CC)$ and $P$ maximal.
In this case, in \cite{algo}, it is proved that the set of triples
$(u,v,w)\in W^P$ such
that $\tilde c_{uv}^w=1$ only depends on $\Phi_u^v$, according to
Conjecture~\ref{conjcomb}.
If $G/P$ is cominuscule $\tilde c_{uv}^w=c_{uv}^w$ for any $(u,v,w)\in
W^P$.
Then the Thomas-Young combinatorial rule \cite{TY} for  $c_{uv}^w$
implies that Conjecture~\ref{conjcomb} holds.

\bigskip
\noindent
{\bf Distributions and  Schubert varieties.}
In Section~\ref{sec:schubertvar}, we study the restriction of the distributions to 
the Schubert varieties $X_u$. More precisely,  for any $x$ in $X_u$ and
$\alpha\in X(Z)$ we are interested in $T_xX_u\cap T_x^{\ggeq\alpha}G/P$.
For $\alpha$ fixed, the dimension of  $T_xX_u\cap T_x^{\ggeq\alpha}G/P$ has a
fixed value for $x\in X_u$ general  and can jump for $x$
in a strict subvariety of $X_u$. 
Consider  the maximal open subset $X_u^0$ of $X_u$ such that 
for any $\alpha\in X(Z)$ 
the dimension of  $T_xX_u\cap T^{\ggeq\alpha}_xG/P$ does not depend on $x\in
X_u^0$.
Consider the global stabilizer $Q_u$ of $X_u$, that is, the set of
$g\in G$ such that $g.X_u=X_u$. Since $X_u$ is $B$-stable, $Q_u$ is a
standard parabolic subgroup of $G$. 

\begin{prop}
\label{prop:flisseSchub} 
 With above notation, we have
$$
X_u^0=Q_u.uP/P.
$$
\end{prop}

If $G/P$ is cominuscule, the filtration is trivial and
Proposition~\ref{prop:flisseSchub} asserts that $Q_u$ acts transitively on
the smooth locus of $X_u$.
This was previously proved by Brion and Polo in \cite{BP:largeschubvar}. 
Proposition~\ref{prop:flisseSchub}  is in the philosophy to generalize known
results from cominuscule homogeneous spaces to any homogeneous space $G/P$, using the
Belkale-Kumar product.  

Note that Proposition~\ref{prop:flisseSchub}  is equivalent to
\cite[Theorem~7.4]{BKR}. Nevertheless, we think that the
distributions give a pleasant interpretation of this result. 
In Section~\ref{sec:schubertvar} we present a proof using the properties 
of the Peterson map.

Retruning to the setting of Conjecture~\ref{conj1}, we assume moreover
that the intersection $u\inv X_u\cap w_0^Pv\inv X_v$ is proper. 
Then Conjecture~\ref{conj1} is implied by the fact that
$\Sigma_u^v$ is the only irreducible component of this intersection
that has the same $X(Z)$-dimension (see Section~\ref{sec:Gammadim}). 
Proposition~\ref{prop:flisseSchub} is clearly related to this version
of Conjecture~\ref{conj1}.

\bigskip
\noindent
{\bf Acknowledgment.}
I thank M. Herzlich, P.E. Paradan,  N. Perrin, C. Vernicos for useful discussions.
The author is partially supported by the French National Agency
(Project GeoLie ANR-15-CE40-0012) and the Institut Universitaire de
France (IUF).

\tableofcontents

\section{Infinitesimal filtrations}

\subsection{The case of a vector space}

\noindent{\bf Ordered group.}
Let $\Gamma$ be a finitely generated free  abelian group whose the law is denoted by
$+$. 
Consider the vector space $\Gamma\otimes_\ZZ\QQ$.
Assume that  a closed strictly convex cone $\cone$  in $\Gamma\otimes_\ZZ\QQ$ of
nonempty interior is given.
Moreover, we assume that $\cone$ is rational polyhedral, that is defined
by finitely many linear rational inequalities, or equivalently
generated, as a cone,  by finitely many vectors in $\Gamma\otimes_\ZZ\QQ$.
We consider the partial order $\lleq$ on $\Gamma$ defined by
$\alpha\lleq\beta$ if and only if $\beta-\alpha$ belongs to $\cone$.
The group $\Gamma$ endowed with
the order $\lleq$ is an  ordered group:
\begin{eqnarray}
  \label{eq:1}
  \forall \alpha,\,\beta,\,\gamma\in \Gamma\ \ \ \alpha\lleq \beta\
  \Rightarrow\ 
(\alpha+\gamma)\lleq (\beta+\gamma).
\end{eqnarray}

The order $\lleq$ satisfies the following version of the Ramsey 
theorem (see also Bolzano-Weirstrass' theorem).

\begin{lemma}\label{lem:BW}
Let $(\alpha_n)_{n\in \NN}$ be a sequence of pairwise distinct
elements of $\Gamma$ such that $\alpha_n\lleq 0$ for any $n$.

Then there exists a subsequence $(\alpha_{\phi(n)})_{n\in \NN}$
such that for any $n$
$$
\alpha_{\phi(n+1)}\slleq\alpha_{\phi(n)}.
$$  
\end{lemma}

\begin{proof}
Let $\varphi_1,\dots,\varphi_s$ be elements of $\Hom(\Gamma,\ZZ)$ such
that $x\in \cone$ if and only if $\varphi_i(x)\geq 0$ for any $i=1,\dots,s$.

Consider first the sequence $\varphi_1(\alpha_n)$ and   
set 
$$
I_1=\{n\,|\,\forall m\geq n\ \ \varphi_1(\alpha_m)>\varphi_1(\alpha_n)\}.
$$
Assume, for a contradiction that  $I_1$ is infinite. 
Denoting by  $\phi(k)$ the $k^{th}$ element
of $I_1$, we get an increasing subsequence of
$(\varphi_1(\alpha_n))_{n\in\NN}$. But $\varphi_1(\alpha_n)\in\ZZ$ and
$\varphi_1(\alpha_n)\leq 0$: a contradiction.
Hence $I_1$ is finite.

Up to taking a subsequence, we may assume that  $I_1$ is empty; that is
$$
\forall n\geq 0\qquad \exists m>n \qquad \varphi_1(\alpha_m)\leq\varphi_1(\alpha_n).
$$
This property allows to construct a nonincreasing subsequence of
$\varphi_1(\alpha_n)$.
Hence, by considering such a subsequence, we may assume that
$$
\forall n\geq 0\qquad
\varphi_1(\alpha_{n+1})\leq\varphi_1(\alpha_n).
$$ 
By successively proceeding similarly, for $i=2,\dots,s$, one gets a
subsequence $\alpha_{\psi(n)}$ such that
$$
\forall i=1,\dots,n\qquad\forall n\qquad
\varphi_i(\alpha_{\psi(n+1)})\leq\varphi_1(\alpha_{\psi(n)}).
$$
Since the $\alpha_n$ are pairwise distinct, we deduce that $\alpha_{\psi(n+1)}\lleq\alpha_{\psi(n)}$.
\end{proof}

\begin{remark}
  Consider the cone $\cone=\{(x,y)\in\RR^2\,:\, y\geq 0 {\rm\ and\ }
  \sqrt 2 x-y\geq 0\}$ and the group $\Gamma=\ZZ^2$. Lemma~\ref{lem:BW}
  does not hold for the induced order $\ggeq$ showing the
  rationality assumption on $\cone$ is necessary.
Indeed, denote by $\pi\,:\,\RR^2\longto\RR$ the linear projection on
the line $y=0$ with kernel the line $y=\sqrt 2x$. 
Then $\pi(\ZZ^2)$ is dense as the group generated by $1$ and
$\frac{\sqrt 2}2$. 
In particular, one can construct a sequence $(x_n,y_n)_{n\in\NN}$ such
that $y_{n+1}<y_n< 0$ and  $0>\sqrt 2 x_{n+1}-y_{n+1}>\sqrt 2
x_n-y_n$. Then the elements of the sequence are pairwise incomparable
for the partial order $\ggeq$.  
\end{remark}

\bigskip
\noindent{\bf $\Gamma$-filtration.} The group $\Gamma$ is used here to
index filtrations.\\

\begin{defin}
Let $V$ be a finite dimensional real or complex vector space.
  A {\it  $\Gamma$-filtration} of $V$ is a collection
  $F^{\ggeq \beta}V$ of linear subspaces of $V$ indexed by
  $\beta\in\Gamma$ satisfying
  \begin{enumerate}\item
\label{ax:dec}   
    $\alpha\lleq\beta\ \Rightarrow\  F^{\ggeq\beta}V\subset
    F^{\ggeq\alpha}V$,
  \item \label{ax:tot} $\exists \beta_0\in\Gamma {\rm \ s.t.\ } V=F^{\ggeq\beta_0}V$,
\item \label{ax:neg} if $F^{\ggeq\alpha}V\neq \{0\}$ then $\alpha\lleq 0$.
 \end{enumerate}
\end{defin}

\begin{lemma}
\label{lem:fini}
  Let $(F^{\ggeq\beta}V)_{\beta\in\Gamma}$ be a $\Gamma$-filtration.
Then the set $\{F^{\ggeq\beta}V\,|\,\beta\in\Gamma\}$ of linear subspaces of
$V$ is finite.
\end{lemma}

\begin{proof}
By contradiction, assume that there exists a sequence
$F^{\ggeq\alpha_n}V$ of pairwise distinct linear subspaces of $V$.
By axiom~\ref{ax:neg}, $\alpha_n\lleq 0$ for any but eventually one
$n$.
Now, Lemma~\ref{lem:BW} implies that there exists a decreasing
subsequence $\alpha_{\phi(k)}$.
Since the linear subspaces $F^{\ggeq\alpha_n}V$ are pairwise distinct, the subsequence
$F^{\ggeq\alpha_{\phi(k)}}V$ is increasing.
This contradicts the assumption that $V$ is finite dimensional.
\end{proof}

\bigskip
\noindent{\bf $\Gamma$-filtrations coming from decompositions.}
For each $\beta\in\Gamma$,  
$\sum_{\alpha\sggeq \beta}F^{\ggeq \alpha}V$ is a linear subspace of
$F^{\ggeq \beta}V$. 
Let us choose a supplementary subspace $S^\beta$:
\begin{eqnarray}
  \label{eq:8}
  F^{\ggeq \beta}V=S^\beta\oplus \sum_{\alpha\sggeq \beta}F^{\ggeq
  \alpha}V.
\end{eqnarray}

One of the motivation for axiom~\ref{ax:neg} in the definition of
$\Gamma$-filtration is the following lemma.

\begin{lemma}\label{lem:fil1}
  With above notation, 
\begin{eqnarray}
  \label{eq:9}
  F^{\ggeq \beta}V=\sum_{\alpha\ggeq \beta}S^\alpha.
\end{eqnarray}
\end{lemma}

\begin{proof}
It is clear that the sum is contained in $F^{\ggeq \beta}V$. Conversely,
since $V$ is finite dimensional, we have 
$$
F^{\ggeq \beta}V=S^\beta\oplus (F^{\ggeq
  \alpha_1}V+\cdots+F^{\ggeq
  \alpha_s}V),
$$
for some
$\alpha_i\in\Gamma$ such that $\alpha_i\sggeq\beta$.
By axiom~\ref{ax:neg}, we may assume that
for any $i=1,\dots,s$ we have $\alpha_i\lleq 0$.
If each $F^{\ggeq  \alpha_i}V$ satisfies the lemma, the lemma is
proved for $F^{\ggeq \beta}V$.
Otherwise, 
we restart the proof with each $\alpha_i$ in place of $\beta$.
Since $\Gamma$ is discrete, the set of $\alpha\in\Gamma$ such that
$0\ggeq\alpha\ggeq\beta$ is finite. In particular,
the procedure ends by axiom~\ref{ax:neg} of the definition of a $\Gamma$-filtration.
\end{proof}

\bigskip
Conversely, assume that a linear subspace $S^\alpha$ of $V$ is given 
for any
$\alpha\in\Gamma$.
If these linear subspaces satisfy $(S^{\alpha}\neq
\{0\}\,\Rightarrow\,\alpha\lleq 0)$, and there exist $\alpha_1,\dots,\alpha_s$ such that
$V=S^{\alpha_1}+\cdots+ S^{\alpha_s}$ then the
formula~\eqref{eq:9} defines a 
$\Gamma$-filtration of $V$.
The $\Gamma$-filtration of $V$ is said to {\it come from a decomposition} if there
exists a decomposition
\begin{eqnarray}
  \label{eq:10}
  V=\bigoplus_{\alpha\in\Gamma}S^\alpha,\mbox{ with } S^{\alpha}\neq \{0\}\,\Rightarrow\,\alpha\lleq 0,
\end{eqnarray}
such that  \eqref{eq:9} holds.

\bigskip
The {\it $f$-dimension vector} ($f$ stand for filtered) of the $\Gamma$-filtration, is the vector
$(fd^\beta(V))_{\beta\in\Gamma}$  of
$\NN^\Gamma$ defined by
$$\Gamma\longto\NN,\,\beta\longmapsto fd^\beta(V)=\dim(F^{\ggeq \beta}V),$$
for any $\beta\in\Gamma$.
Define the {\it grading associated to the 
$\Gamma$-filtration} by setting
\begin{eqnarray}
  \label{eq:3}
  \G^\beta V=\frac{F^{\ggeq \beta}V}{\sum_{\alpha\sggeq \beta}F^{\ggeq \alpha}V},
{\rm\ and\ }
\G V=\bigoplus_{\beta\in\Gamma} \G^\beta V.
\end{eqnarray}
The {\it $g$-dimension vector} ($g$ stands for graded) $(gd^\beta(V))_{\beta\in\Gamma}$ of the $\Gamma$-filtration is 
defined by 
$$\Gamma\longto\NN,\,\beta\longmapsto gd^\beta(V):=\dim(\G^\beta V).$$

\begin{lemma}
 \label{lem:decdim}
The $\Gamma$-filtration comes
from a decomposition if and only if 
\begin{eqnarray}
  \label{eq:11}
\dim(\G V)=\dim(V).
\end{eqnarray}
In this case, the $g$-dimension vector of $V$ only depends on the 
$f$-dimension vector of $V$.
\end{lemma}

\begin{proof}
Assume first that the $\Gamma$-filtration comes from a decomposition.
  Fix linear subspaces $S^\alpha$ satisfying the conditions~\eqref{eq:10}
  and \eqref{eq:9}. For any $\beta\in\Gamma$,  the identity~\eqref{eq:8}
  holds and $\dim (\G^\beta V)=\dim(S^\beta)$.
Hence the lemma follows from  the condition~\eqref{eq:10}.\\

Conversely, assume that the condition~\eqref{eq:11} is fulfilled
and choose linear subspaces $S^\beta$ satisfying~\eqref{eq:8}.
Let $\beta_0\in\Gamma$ such that $F^{\ggeq\beta_0}V=V$.
Lemma~\ref{lem:fil1} implies that  $V=\sum_{\beta\ggeq\beta_0} S^\beta$.
The condition~\eqref{eq:11} implies that the sum is direct. 
Moreover, it implies that $S^\gamma=\{0\}$ if
$\gamma\not\ggeq\beta_0$.
Using Lemma~\ref{lem:fil1} once again, we deduce that the filtration
comes from the decomposition $V=\bigoplus_\beta S^\beta$.
\\

If $fd^\beta=0$ then $F^{\ggeq \beta}V=\{0\}$,  $\G^\beta=\{0\}$
and $gd^\beta=0$.
Let $\Gamma_{max}$ be the set of maximal elements among the elements
$\beta$ in $\Gamma$ satisfying $F^{\ggeq \beta}V\neq\{0\}$.
For $\beta\in \Gamma_{\max}$, we have $gd^\beta(V)=fd^\beta(V)$. 
Et caetera.
\end{proof}

\bigskip
\begin{exple}
Consider the group $\ZZ^2$ endowed with the order $(a,b)\lleq (a',b')$
if and only if $a\leq a'$ and $b\leq b'$.
Fix a two dimensional vector space $V$ and three pairwise distinct lines
$l_1$, $l_2$, and $l_3$ in $V$.
Consider the following family $(S^\beta)_{\beta\in\ZZ^2}$ of
linear subspaces of $V$:  
$
S^{(-2,0)}=l_1,
$  
$
S^{(0,-2)}=l_2,
$  
$
S^{(-1,-1)}=l_3,
$  
and $S^\beta=\{0\}$ if $\beta\not\in\{(-2,0),(0,-2),(-1,-1)\}$.
The  filtration defined by the formula~\eqref{eq:9} does not come from a
decomposition.
More precisely, $\G V\simeq l_1\oplus l_2\oplus l_3$ has dimension
three whereas $V$ has dimension two. 
\end{exple}

\bigskip
Another useful notion is the {\it weight} $\rho(V)$ of the $\Gamma$-filtration of $V$
defined by
\begin{eqnarray}
  \label{eq:30}
 \rho(V)=\sum_{\beta\in\Gamma}gd^\beta(V)\beta. 
\end{eqnarray}

\bigskip
\noindent{\bf Filtrations induced on a linear  subspace.}
Let $W$ be a linear subspace of $V$. The $\Gamma$-filtration on $V$
induces one on $W$ by setting
\begin{eqnarray}
  \label{eq:29}
  \forall \beta\in\Gamma\ \ \ F^{\ggeq\beta}W:=W\cap F^{\ggeq \beta}V.
\end{eqnarray}

\begin{lemma}
  \label{lem:decsous}
If the $\Gamma$-filtration on $V$ comes
from a decomposition then the induced $\Gamma$-filtration on $W$ comes
from a decomposition.
\end{lemma}

\begin{proof}
Fix linear subspaces $S_V^\beta$ and $S_W^\beta$ of $V$
such that
$$
V=S^\beta_V\oplus F^{\ggeq\beta}V,\ \ \ 
W=S^\beta_W\oplus F^{\ggeq\beta}W,\ \ \ 
S_W^\beta\subset S_V^\beta.
$$ 
 Lemma~\ref{lem:fil1} implies that
\begin{eqnarray}
\label{eq:sum2}
W=\sum_{\beta\in\Gamma} S^\beta_W.
\end{eqnarray}
 Lemma~\ref{lem:decdim} shows 
$$
V=\bigoplus_{\beta\in\Gamma} S^\beta_V.
$$
Since $S_W^\beta\subset S_V^\beta$ it follows that the sum
\eqref{eq:sum2} is direct.
\end{proof}

\bigskip
\noindent{\bf Filtrations induced on $p$-forms.}
Let $p$ a nonnegative integer. A $\Gamma$-filtration of $V$ induces  a filtration on the space
$\bigwedge^pV^*$ of skewsymmetric $p$-forms on $V$ as follows.\\

\begin{defin}
Let $\beta\in\Gamma$.
Denote by  $F^{\lleq \beta}\bigwedge^pV^*$  the linear subspace of forms
$\omega\in \bigwedge^pV^*$ such that 
for any any $\alpha_1,\dots,\alpha_p\in\Gamma$, for any $v_i\in F^{\ggeq\alpha_i}V$, we have
 \begin{eqnarray}
   \label{eq:25}
   \alpha_1+\cdots+\alpha_p\nlleq \beta \Rightarrow
\omega(v_1,\dots,v_p)=0.
 \end{eqnarray}
\end{defin}

The first properties of these linear subspaces are.

\begin{prop}
 \label{prop:filform} 
 \begin{enumerate}
 \item  If $\beta\lleq \gamma$ then $F^{\lleq\beta}\bigwedge^pV^*\subset
    F^{\lleq\gamma}\bigwedge^pV^*$.
\item Let $\beta_0\in\Gamma$ be such that $F^{\ggeq\beta_0}V=V$. 
If $F^{\lleq\gamma}\bigwedge^pV^*\neq\{0\}$ then 
$\gamma\ggeq p\beta_0$.
\item We have $F^{\lleq 0}\bigwedge^pV^*=\bigwedge^pV^*$.
\item For  $\beta$ and $\gamma$ in $\Gamma$, we have
$F^{\lleq\beta}\bigwedge^pV^*\wedge
    F^{\lleq\gamma}\bigwedge^qV^*\subset F^{\lleq
      \beta+\gamma}\bigwedge^{p+q}V^*$.

 \end{enumerate}
\end{prop}

\begin{proof}
  If $\beta\lleq \gamma$ then $\alpha_1+\cdots+\alpha_p\nlleq \gamma$
  implies $\alpha_1+\cdots+\alpha_p\nlleq \beta$. Hence the conditions
  defining $F^{\lleq\gamma}\bigwedge^pV^*$ are  conditions defining $F^{\lleq\beta}\bigwedge^pV^*$.
 The first assertion follows.

Let $\gamma\nggeq p\beta_0$.
The definition of $F^{\lleq\gamma}\bigwedge^pV^*$ with
$\alpha_1=\cdots=\alpha_p=\beta_0$
implies that $F^{\lleq\gamma}\bigwedge^pV^*$  is reduced to zero.

Let $\omega$ be any $p$-form.
We want to prove that $\omega\in  F^{\lleq 0}\bigwedge^pV^*$.
Let $\alpha_1,\dots,\alpha_p$ such that
$\alpha_1+\cdots+\alpha_p\nlleq 0$.
Then some $i_0$ satisfies $\alpha_{i_0}\nlleq 0$. 
In particular, $F^{\ggeq\alpha_{i_0}}V=\{0\}$.
This implies that $\omega$ is zero on
$F^{\ggeq\alpha_{1}}V\times\cdots\times F^{\ggeq\alpha_s}V$.

\bigskip
Let $\omega_1$ and $\omega_2$ belong to $F^{\lleq\beta}\bigwedge^pV^*$ and 
   $ F^{\lleq\alpha}\bigwedge^qV^*$ respectively.
Let $\alpha_1,\dots,\alpha_{p+q}$ be such that
$\alpha_1+\cdots+\alpha_{p+q}\nlleq \beta+\gamma$.
Let $v_i\in F^{\ggeq\alpha_i}V$.
Then
\begin{eqnarray}
  \label{eq:26}
  \begin{array}{l@{}l}
 
  (\omega_1\wedge\omega_2)&(v_1,\dots,v_{p+q})=\\[1em]&
\frac 1 {(p+q)!}\sum_{\sigma\in\Sym_{p+q}}\varepsilon(\sigma)
\omega_1(v_{\sigma(1)},\dots,v_{\sigma(p)}). \omega_2(v_{\sigma(p+1)},\dots,v_{\sigma(p+q)}).
  \end{array}\end{eqnarray}
It is sufficient to prove that any term in  the sum~\eqref{eq:26} is
zero.
Since $(\alpha_{\sigma(1)}+\cdots
+\alpha_{\sigma(p)})+(\alpha_{\sigma(p+1)}+\cdots+\alpha_{\sigma(p+q)})
\nlleq \beta+\gamma$, either $(\alpha_{\sigma(1)}+\cdots
+\alpha_{\sigma(p)}) \nlleq \beta$ or $(\alpha_{\sigma(p+1)}+\cdots+\alpha_{\sigma(p+q)})
\nlleq \gamma$.
In the two cases, the product
$$\omega_1(v_{\sigma(1)},\dots,v_{\sigma(p)}). \omega_2(v_{\sigma(p+1)},\dots,v_{\sigma(p+q)})$$
is equal to zero.
\end{proof}

\bigskip
\begin{remark}
The three first assertions of Proposition~\ref{prop:filform} mean that $(F^{\lleq \beta}\bigwedge^pV^*)_{\beta\in\Gamma}$
is a $\Gamma$-filtration of $\bigwedge^pV^*$ up to the changing of
index
$\beta\mapsto p\beta_0-\beta$.
Indeed, even for $p=1$, taking orthogonal reverses inclusions and
exchanges $\{0\}$ with the whole space.
\end{remark}

\bigskip
\noindent{\bf Filtrations coming from a decomposition.}
\begin{lemma}
  \label{lem:decfrom}
Let $p$ be a positive integer.
If the $\Gamma$-filtration on $V$ comes
from a decomposition then the induced $\Gamma$-filtration $F^{\ggeq p\beta_0-\beta}\bigwedge^pV^*$ on $\bigwedge^pV^*$ comes
from a decomposition.
\end{lemma}

\begin{proof}
Write
$$
V=\bigoplus_{\alpha\in\Gamma}S^\alpha
\quad{\rm and}\quad
F^{\ggeq \beta}V=\bigoplus_{\alpha\ggeq \beta}S^\alpha,
$$ 
with  $(S^{\alpha}V\neq \{0\}\,\Rightarrow\,\alpha\lleq 0)$. 
For any $\beta\in\Gamma$, denote by $T^\beta$ the orthogonal of
$\bigoplus_{\alpha\neq \beta}S^\alpha$ in $V^*$.
It can be identified with the dual of $S^\beta$ and
\begin{equation}
  \label{eq:dualdec}
V^*=\bigoplus_{\beta\in\Gamma}T^\beta.  
\end{equation}
For any collection of  subspaces $F_1,\dots,F_p$ of $V^*$, 
$\pi(F_1\otimes\cdots \otimes F_p)$ denotes the  subspace of
$\wedge^pV^*$ obtained by adding wedge products of elements of the subspaces
$F_i$.
For any $\theta\in\Gamma$,  set 
$$
(\wedge^pV^*)^\theta:=
\sum_{\beta_1+\cdots+\beta_p=\theta}\pi(T^{\beta_1}\otimes\cdots\otimes T^{\beta_p}).
$$
It is clear that \eqref{eq:dualdec} implies that
$$
\wedge^pV^*=\bigoplus_{\theta\in\Gamma}(\wedge^pV^*)^\theta.
$$
Moreover, for any $\theta\in\Gamma$, $(\wedge^pV^*)^\theta$ is the set
of $p$-forms $\omega$ such that for any $\alpha_i\in\Gamma$ and
$v_i\in S^{\alpha_i}$
such that $\alpha_1+\cdots+\alpha_p\neq\theta$ we have
$\omega(v_1,\dots,v_p)=0$.

We claim that 
\begin{eqnarray}
  \label{eq:736}
F^{\lleq\beta}\wedge^pV^*=\bigoplus_{\theta\lleq\beta}(\wedge^pV^*)^\theta.
\end{eqnarray}
Indeed
$F^{\lleq \beta}\bigwedge^pV^*$ is the subspace of forms
$\omega\in \bigwedge^pV^*$ such that 
for any $\alpha_1,\dots,\alpha_p\in\Gamma$, any $v_i\in S^{\alpha_i}$, we have
$$
   \alpha_1+\cdots+\alpha_p\nlleq \beta \Rightarrow
\omega(v_1,\dots,v_p)=0.
 $$

Let $\theta$ such that $(\wedge^pV^*)^\theta\neq\{0\}$.
Then   there exist $\beta_1,\dots,\beta_p$ in $\Gamma$  such that
$\beta_1+\cdots+\beta_p=\theta$ and $T^{\beta_i}\neq\{0\}$ for any
$i$.
Hence $S^{\beta_i}\neq\{0\}$ for any
$i$ and $\beta_i\lleq 0$.  We deduce that $\theta\lleq 0$.
\end{proof}

\subsection{The case of manifolds}
\label{sec:manifolds}

Let $M$ be a smooth connected manifold and let $TM$ denote its tangent
bundle. Here comes the central definition of this work.

\bigskip
\begin{defin}
\label{def:infGammafilt}
  An {\it infinitesimal $\Gamma$-filtration} of $M$ is a collection
  $F^{\ggeq \beta}TM$ of vector subbundles of $TM$ indexed by
  $\beta\in\Gamma$ satisfying
 \begin{enumerate}
  \item 
    \label{eq:2}
    $\alpha\lleq\beta\ \Rightarrow\  
  F^{\ggeq\beta}TM \subset F^{\ggeq\alpha}TM
    $,
  \item $\exists \beta_0\in\Gamma {\rm \ s.t.\ } TM=F^{\ggeq\beta_0}TM$,
\item if $F^{\ggeq\alpha}TM\neq \{\underline{0}\}$ then $\alpha\lleq 0$.
 \end{enumerate}

\end{defin}

\bigskip
The {\it f-rank vector} of the infinitesimal filtration is the map
\begin{eqnarray}
  \label{eq:4}
  \beta\longmapsto \rk( F^{\ggeq \beta}TM),
\end{eqnarray}
and belongs to $\NN^\Gamma$. 

\bigskip
\begin{defin}
  An infinitesimal $\Gamma$-filtration is said to {\it come from a
    decomposition} if for any $x\in M$, the $\Gamma$-filtration of 
$T_xM$ comes from a decomposition.
\end{defin}

\bigskip
\begin{remark}
We do not require a $\Gamma$-decomposition of the tangent bundle $TM$
but only for a punctual decomposition.
\end{remark}

\begin{lemma}
\label{lem:GrM}
Consider an  infinitesimal $\Gamma$-filtration on $M$ coming from a
decomposition.
Then for any $\beta$, the sum
$\sum_{\alpha\sggeq\beta} F^{\ggeq\alpha}TM$ is a subbundle of $TM$.
\end{lemma}

\begin{proof}
  Fix $x$ in $M$ and a $\Gamma$-decomposition of $T_xM=\oplus_\alpha
  S^\alpha$ such that the identities~\eqref{eq:10} and \eqref{eq:9} hold.
Then  $\sum_{\alpha\sggeq\beta}
F^{\ggeq\alpha}T_xM=\sum_{\alpha\sggeq\beta} S^{\alpha}$.
In particular, its dimension only depends on the $g$-dimension vector of the
filtration of $T_yM$.
This $g$-dimension vector only depends on the $f$-dimension by Lemma~\ref{lem:decdim}. 
It follows that the dimension of $\sum_{\alpha\sggeq\beta}
F^{\ggeq\alpha}T_xM$ does not depend on $x$.
Now, the lemma follows from classical properties of vector subdundles.
\end{proof}

\bigskip
Define the {\it grading associated to the infinitesimal
$\Gamma$-filtration coming from a decomposition} by setting
\begin{eqnarray}
  \label{eq:3TM}
  \G^\beta TM=\frac{F^{\ggeq \beta}TM}{\sum_{\alpha\sggeq \beta}F^{\ggeq \alpha}TM}
  \quad
{\rm\ and\ }
\quad
\G TM=\bigoplus_{\beta\in\Gamma} \G^\beta TM.
\end{eqnarray}
They are vector bundles on $M$.
The {\it g-rank vector} $(gd^\beta(M))_{\beta\in\Gamma}$ of the $\Gamma$-filtration is 
defined by 
$$\Gamma\longto\NN,\,\beta\longmapsto gd^\beta(M):=\rk(\G^\beta TM).$$

\subsection{The case of varieties}
\label{sec:Gammadim}
Let $X$ be a smooth complex irreducible variety.
Consider the complex tangent bundle $TX$. 

\bigskip
\begin{defin}
  An infinitesimal $\Gamma$-filtration of $X$ is said to be {\it algebraic} if
  each $F^{\ggeq\beta}TX$ is a complex algebraic vector subbundle of $TX$.
\end{defin}

\bigskip
Let $Y$ be an irreducible subvariety of $X$. For $y\in Y$, the
Zariski-tangent space $T_yY$ of $Y$ at the point $y$ is a complex
subspace of $T_yX$. Set
\begin{eqnarray}
  \label{eq:6}
  F^{\ggeq \beta}T_yY=F^{\ggeq \beta}T_yX\cap T_yY.
\end{eqnarray}

Even if $Y$ is smooth, $F^{\ggeq \beta}T_yY$ does not define a
subbundle of $TY$ since its dimension depends on $y$.   

\begin{lemma}
  \label{lem:dimsc}
For any $\beta\in\Gamma$ and $y\in Y$, there exists an open
neighborhood $U$ of $y'$ in $Y$ such that for any $y'\in U$ we have
\begin{eqnarray}
  \label{eq:7}
  \dim(F^{\ggeq \beta}T_yY)\geq \dim(F^{\ggeq \beta}T_{y'}Y).
\end{eqnarray}
\end{lemma}

\begin{proof}
 Locally in $y\in Y$ the subspace $F^{\ggeq \beta}T_yY$ of $T_yX$
 can be expressed as the kernel of a matrix whose coefficients depends
 algebraically on $y$. The lemma follows.
\end{proof}

\bigskip
The point $y\in Y$ is said
to be {\it $\Gamma$-regular} if 
\begin{eqnarray}
  \label{eq:17}
  \forall\beta\in\Gamma\ \ \dim(F^{\ggeq \beta}T_yY)=\min_{y'\in Y} \dim(F^{\ggeq \beta}T_{y'}Y).
\end{eqnarray}
Since $\Gamma$ is countable, Lemma~\ref{lem:dimsc} shows that
a very general point in $Y$ is $\Gamma$-regular. More precisely, 
Lemma~\ref{lem:fini} implies that the set of $\Gamma$-regular points in $Y$
is open. The open set of $\Gamma$-regular points of $Y$ is denoted by
$Y^{\Gamma-{\rm reg}}$. If $x\in Y^{\Gamma-{\rm reg}}$, the
$g$-dimesnion of $T_xY$ is called the $\Gamma$-dimension of $Y$.

\section{Infinitesimal filtration of $G/P$ and Schubert varieties}
\label{sec:schubertvar}

\subsection{Infinitesimal filtration of $G/P$}
\label{sec:GP}

As in the introduction,  $G$ is a complex semisimple group,   $P$ is a parabolic subgroup
of $G$, $T\subset B\subset P$ are a fixed maximal torus and a Borel subgroup. 
Moreover,  $L$ denotes the Levi subgroup of $P$ containing $T$ and $Z$ denotes
the neutral component of its center. 
The group of multiplicative characters of $Z$ is denoted by $X(Z)$.
Set $\Gamma=X(Z)$. Our main example is an infinitesimal
$X(Z)$-filtration of $G/P$.

Let  $S$ be any torus. 
If $V$ is any $S$-module then 
$\Phi(V,S)$ denotes the set of nonzero weights of $S$ on $V$. 
For $\beta\in X(S)$, $V_\beta$ denotes the eigenspace of weight $\beta$. 
 
Denote by $\lp$ and $\lieg$  the Lie algebras of $P$ and $G$  and consider the convex cone $\cone$ generated by $\Phi(\lp,Z)$ in
$X(Z)\otimes\QQ$.
It is a closed strictly convex polyhedral cone of nonempty interior in  
$X(Z)\otimes\QQ$. 
The associated order on $X(Z)$ is denoted by  $\ggeq$.
The decomposition of $\lieg/\lp$ under the action of $Z$:
\begin{eqnarray}
  \label{eq:818}
\lieg/\lp=\bigoplus_{
 \alpha\in X(Z)
}  (\lieg/\lp)_{\alpha}
\end{eqnarray}
 is supported on $-\cone\cap X(Z)$.
The group $P$ acts on $\lieg/\lp$  by the adjoint action but does not stabilize the
decomposition~\eqref{eq:818}. For any $\beta \in X(Z)$, the linear subspace
\begin{eqnarray}
  \label{eq:19}
  F^{\ggeq\beta}\lieg/\lp=\bigoplus_{
    \begin{array}{c}
 \alpha\in X(Z)\\
\alpha\ggeq\beta
    \end{array}
}  (\lieg/\lp)_{\alpha}
\end{eqnarray}
is $P$-stable.
More precisely, the set of  $F^{\ggeq\beta}\lieg/\lp$ forms a $P$-stable
$X(Z)$-filtration of $\lieg/\lp$ coming from the decomposition~\eqref{eq:818}.
The tangent bundle $T(G/P)$ identifies with the fiber bundle $G\times_P\lieg/\lp$ over $G/P$.
These remarks allow to define a $G$-equivariant infinitesimal $X(Z)$-filtration on
$G/P$ by setting for any $\beta\in X(Z)$
\begin{eqnarray}
  \label{eq:20}
   F^{\ggeq \beta}T(G/P)=G\times_PF^{\ggeq
      \beta}\lieg/\lp.
\end{eqnarray}

Consider the set $\Phi(\lieg/\lp,T)$ of weights of $T$ acting on
$\lieg/\lp$. Then  $\Phi(\lieg/\lp,T)$ is a subset of $\Phi$.
Let $w$ belong to $W^P$ and consider the centered Schubert variety
$w^{-1}X_w$.
Then $P/P$ belongs to the open $w^{-1}Bw$-orbit in  $w^{-1}X_w$. 
In particular, it is $X(Z)$-regular.
Denote by  $\Phi(w)$ the set of weights of $T$ acting on
$T_{P/P}w^{-1}X_w$. Then $\Phi(w)=\Phi^-\cap w^{-1}\Phi^+$ is the
inversion set of $w$. Moreover, 
$\Phi(w)$ is contained in $\Phi(\lieg/\lp,T)$. 
Since $P/P$ is $X(Z)$-regular in $w\inv X_w$, the $g$-dimension of $X_w$ is equal to the
 $g$-dimension of $T_{P/P}w^{-1}X_w$. The following result follows
directly:

\begin{lemma}
 \label{lem:gdimSchub}
The $g$-dimension of $gd(X_w)$ of $X_w$ is equal to
$$
\begin{array}{ccl}
  X(Z)&\longto&\ZZ_{\geq 0}\\
\alpha&\longmapsto&\#\{\theta\in\Phi(w)\,:\,\theta_{|Z}=\alpha\},
\end{array}
$$
where $\theta$ belongs to $X(T)$ and $\theta_{|Z}$ denotes its
restriction to $Z$.
\end{lemma}
\subsection{Peterson's application}

Let $V'$ be any $T$-module without multiplicity and let $\beta\in X(T)$.
Under the action of $\Ker\beta\subset T$, $V'$ decomposes 
\begin{eqnarray}
  \label{eq:22}
  V'=\bigoplus_{\alpha\in X(T)/\ZZ\beta}\left(
\oplus_{k\in\ZZ}V'_{\alpha+k\beta}
\right).
\end{eqnarray}
A subset $\Lambda$ of $\Phi(V',T)$ is said to be {\it $\beta$-convex} if
\begin{eqnarray}
  \label{eq:23}
  \alpha\in \Lambda,\ \alpha+\beta\in\Phi(V',T)\Rightarrow
\alpha+\beta\in \Lambda.
\end{eqnarray}
For any submodule $V$ of $V'$,  $V^\beta$ denotes the unique
sub-$T$-module of $V'$ isomorphic to $V$ as a $\Ker(\beta)$-module 
and such that $\Phi(V^\beta,T)$ is $\beta$-convex.
In other words, on each line $\alpha+\ZZ\beta\cap\Phi(V',T)$, one pushes
the elements of $\Phi(V,T)$ in the direction $\beta$ to get $\Phi(V^\beta,T)$.

\bigskip
Let $w\in W$. 
The point $wP/P$ is denoted by $\dot w$.
Let $V$ be a $T$-submodule of $T_{\dot w}G/P$. 
Let $\beta$ be a root of $(G,T)$.
We are interested in the action of  the 
unipotent one-parameter subgroup $U_\beta$ associated to $\beta$ on
$\dot w$ and $V$.
Consider the
point $\dot v=\lim_{\tau\to\infty}U_\beta(\tau)\dot w$.
For any $\tau\in\CC$, $U_\beta(\tau)V$ is a linear subspace of
$T_{U_\beta(\tau)\dot w}G/P$ of the same dimension as $V$.
 Hence it is a point of a bundle in Grassmannian over $G/P$. 
Consider the limit in this bundle
\begin{eqnarray}
  \label{eq:21}
  \tau(V,\beta):=\lim_{\tau\to \infty} U_\beta(\tau)V.
\end{eqnarray}
This limit  $\tau(V,\beta)$  is a $T$-stable submodule
of the $T$-module without multiplicity $T_{\dot v}G/P$.

We can now state a Peterson's result (see \cite[Section
8]{CK:Petersonmap}).

\begin{lemma}
\label{lem:Peterson}
The $T$-submodule $s_\beta\tau(V,\beta)$ of $T_{\dot w}G/P$ is equal
to $V^{-\beta}$.  
\end{lemma}

\begin{proof}
The set $\{U_\beta(\tau)\dot w\::\:\tau\in\CC\} \cup\dot v$ is a
$T$-stable curve isomorphic to $\PP^1$.
The computation of  $\tau(V,\beta)$ lies in a bundle in
Grassmannians over this line. 
This computation can be made quite explicitly by
trivializing this bundle on the two $T$-stable open affine subsets of
$\PP^1$.
\end{proof}

\subsection{A lemma on $T$-varieties}

The following result is used in this paper to characterize Schubert
varieties in terms of their tangent spaces among the irreducible $T$-stable
subvarieties of $G/P$.

\begin{lemma}
\label{lem:Tvar}
Let $V$ be a  $T$-module.
Let  $\cone$ be a  strictly convex cone in 
  $X(T)\otimes\QQ$.

Let $\Sigma$ be a closed $T$-stable subvariety of $V$ such that
\begin{enumerate}
\item $\Sigma$ is smooth at  $0$;
\item $T_0\Sigma=\oplus_{\chi\in \cone} V_\chi$ .
\end{enumerate}

Then $\Sigma=\oplus_{\chi\in \cone} V_\chi$.
\end{lemma}

\begin{proof}
  Since $\cone$ is  strictly convex and $\Phi(V,T)$ is finite, there exist finitely
  many one-parameter subgroups  $\lambda_1,\dots,\lambda_k$ of $T$
  such that 
$$
\forall\chi\in X(T)\ \ \ 
\chi\in \cone\iff\forall i\ \ \langle\lambda_i,\chi\rangle>0.
$$

For any $i$, there exists a  $T$-stable neighborhood  of $0$ in
$\Sigma$ such that any point  $x$ in this neighborhood satisfies 
 $\lim_{t\to  0}\lambda_i(t)x=0$. 
Consider the set $W$ of $v\in V$ such that $\lim_{t\to
  0}\lambda_i(t)v=0$, for any $i$.
By the second condition, $W$ is precisely $T_0\Sigma$. 
We just proved that $T_0\Sigma$  contains an open subset of $\Sigma$.
But these two varieties are irreducible and of same dimension (since
$\Sigma$ is assumed to be smooth at $0$). Hence 
$\Sigma=T_0\Sigma$.
\end{proof}

\subsection{Schubert varieties}
\label{sec:schubvar}

 Let $Y$ be a subvariety of $G/P$. Let $G(X)$ denote the stabilizer of 
$Y$ in $G$; it is the set of  $g$ in $G$ such that
$gY=Y$.  If $G(Y)$ has an open orbit  in $Y$ then this orbit is called the
{\it homogeneous locus} of $Y$; otherwise, the homogeneous locus of
$Y$ is defined to be empty.
In other words, 
the homogeneous locus of $Y$ is the biggest open subset  of $Y$
homogeneous under a subgroup of $G$; it  is  denoted by $Y^{\rm hom}$.

Recall that $X_w=\overline{BwP/P}$.
If $Y=X_w$ (for some $w\in W^P$) then the group  $G(X_w)$ contains $B$: it is a 
standard parabolic subgroup of $G$. In particular, it is characterized
by a subset $\Delta_w$
of  simple roots. Precisely we set
$$
\Delta_w=\{\alpha\in\Delta\ :\ P_\alpha X_w=X_w\}.
$$

\begin{prop}
 \label{prop:freghom} 
We have
$$
X_w\Greg=X_w^{\rm hom}.
$$
\end{prop}

\begin{proof}
  Since the infinitesimal filtration is $G$-invariant, it is clear that  
$X_w\Greg$ is $G(X_w)$-stable and contains $X_w^{\rm hom}$. 
Moreover Lemma~\ref{lem:dimsc} implies that $X_w\Greg$ is open in $X_w$.

Assume that $X_w\Greg-X_w^{\rm hom}$ is nonempty.
Choose an open $B$-orbit in $X_w\Greg-X_w^{\rm hom}$ and a
$T$-fixed point $\dot v$  on it.

Obviously $v$ is smaller than $w$ for the Bruhat order.
Since the Bruhat order is generated by  $T$-stable curves, there
exists a positive root $\beta$ such that $s_\beta v\in W^P$ and
$v<s_\beta v<w$. Since $B.\dot v$ is dense in an irreducible component of 
$X_w\Greg-X_w^{\rm hom}$, $s_\beta \dot v$ belongs to $X_w^{\rm hom}$.

Since $s_\beta \dot v$ is a $T$-fixed point in $G(X_w).\dot w$, it is equal
to $u\dot w$ for some $u\in W(G(X_w))$.

We claim that
\begin{eqnarray}
  \label{eq:13}
  s_\beta\in G(X_w)/T. 
\end{eqnarray}
Let us first explain how the claim leads to  a contradiction.
Since  $u$ belongs to $G(X_w)/T$, the claim implies that $s_\beta
u^{-1}X_w^{\rm hom}=X_w^{\rm hom}$.
But $\dot v=s_\beta u^{-1} \dot w$ and $\dot w$ belongs to $X_w^{\rm hom}$.
Hence $\dot v\in X_w^{\rm hom}$ which is a contradiction.

\bigskip
Consider $\gamma=\pm w^{-1}u\beta$ where the sign is chosen to make
$\gamma$ 
negative.
Since $u\in G(X_w)/T$,
Claim~\eqref{eq:13} is equivalent to $s_\beta u^{-1}X_w=u^{-1}X_w$
or  to $s_{u\beta}X_w=X_w$ or to
\begin{eqnarray}
  \label{eq:12}
  s_{\gamma}.(w^{-1}X_w)=w^{-1}X_w.
\end{eqnarray}
Look these two varieties in a neighborhood of 
$P/P$.
More precisely, consider the unique affine open $T$-stable neighborhood
$\Omega$ of $P/P$ in $G/P$.
Then $\Omega$ is 
isomorphic as a $T$-variety to a $T$-module without multiplicity.
Since the two varieties of \eqref{eq:12} are irreducible, it is 
 sufficient to prove that 
 \begin{equation}
   \label{eq:Omega}
   \Omega\cap s_{\gamma}.(w^{-1}X_w)=\Omega\cap
 w^{-1}X_w.
 \end{equation}

Since $s_\gamma P/P\in w\inv X_w$, $\gamma\in\Phi(w)$. 
In particular, $w\inv X_w$ is $U_\gamma$-stable. 
But, $s_\gamma P/P$ and $P/P$ are smooth points in $w\inv X_w$.
Hence 
$$
\lim_{\tau\to\infty}U_\gamma(\tau)T_{P/P}w\inv X_w=T_{s_\gamma
  P/P}w\inv X_w.
$$
Then
Lemma~\ref{lem:Peterson} shows that 
$$
\begin{array}{ll}
  \label{eq:15}
  \Phi(T_{P/P}s_\gamma w^{-1}X_w,T)&=s_\gamma\Phi(T_{s_\gamma
    P/P}w^{-1}X_w,T)\\
&=s_\gamma\Phi\left(
\lim_{\tau\to\infty}U_\gamma(\tau)T_{P/P}w^{-1}X_w,T
\right)\\
&=\Phi\left(
(T_{P/P}w^{-1}X_w)^{-\gamma},T
\right).
\end{array}
$$
Since $P/P$ is $\Gamma$-regular in $s_\gamma w\inv X_w$,
\begin{eqnarray}
  \label{eq:16}
  \forall \alpha\in X(Z)\quad 
\dim(F^{\ggeq \alpha}(T_{P/P}w^{-1}X_w) ^{-\gamma})=\dim(F^{\ggeq \alpha}(T_{P/P}w^{-1}X_w)).
\end{eqnarray}
But $\gamma\not\in\Phi(P)$, hence $\gamma_{|Z}$ is non trivial. 
Then, equality~\eqref{eq:16} implies that $\Phi((T_{P/P}w\inv
X_w)^{-\beta},T)=\Phi((T_{P/P}w\inv X_w),T)$.
Equality~\eqref{eq:Omega} follows and the theorem is proved.
\end{proof}

\section{Infinitesimal filtration and cohomology}

\subsection{Filtration of differential forms on a manifold}

In this subsection, $M$ is a smooth connected manifold of dimension
$d$
endowed with an infinitesimal $\Gamma$-filtration. 
The notion that allows to control the differential relatively to the
filtration is the following one.

\bigskip
\begin{defin}
An infinitesimal $\Gamma$-filtration of $M$ is said to be {\it
  integrable} if for any $\alpha$ and $\beta$ in $\Gamma$ we have
\begin{eqnarray}
  \label{eq:24}
   [F^{\ggeq\alpha}T M, F^{\ggeq\beta}T M]\subset
  F^{\ggeq\alpha + \beta}TM.
\end{eqnarray} 
\end{defin}

\begin{exple}
  Let $L$ be an integrable  distribution on $M$.
We get an integrable infinitesimal  $\ZZ$-filtration be setting
$$
\begin{array}{ll}
  F^{\ggeq a}TM=TM&\forall a\in\ZZ_{<0},\\
 F^{\ggeq 0}TM=L,\\
 F^{\ggeq a}TM=\underline{0}&\forall a\in\ZZ_{>0}.\\
\end{array}
$$
\end{exple}

\begin{exple}
  Let $L$ be any distribution on $M$.
We get an integrable infinitesimal  $\ZZ$-filtration be setting 
$$
\begin{array}{ll}
  F^{\ggeq a}TM=TM&\forall a\leq -2,\\
 F^{\ggeq -1}TM=L,\\
 F^{\ggeq a}TM=\underline{0}&\forall a\in\ZZ_{\geq 0}.\\
\end{array}
$$
\end{exple}

\bigskip
Consider the sheaf $\AAf^p$ of differential  $p$-forms on $M$ and the
De Rham differential
$d_p\,:\,\AAf^p\longto\AAf^{p+1}$. The De Rham cohomology group is

$$
{\rm H}_{DR}^p(M,\RR):=\frac{\Ker\; d_p(M)}{{\rm Im}\;d_{p-1}(M)}.
$$ 
The exterior product
$$
\begin{array}{cccc}
  \wedge\,:&\AAf^p\times \AAf^{p'}&\longto&\AAf^{p+p'}\\
&(\omega,\omega')&\longmapsto&\omega\wedge\omega'
\end{array}
$$
induces a product  $\wedge$ in cohomology since 
$$
d(\omega\wedge\omega')=(d\omega)\wedge\omega'+(-1)^p\omega\wedge d\omega'.
$$
In particular, ${\rm H}^*_{DR}(M,\RR):=\oplus_{k=0}^d {\rm H}^k_{DR}(M,\RR)$ is 
a graded algebra.

\bigskip
We now consider the $\Gamma$-filtration on the sheaf $\AAf^p$ induces
by the infinitesimal $\Gamma$-filtration. 

\bigskip
\begin{defin}
Let $\beta\in\Gamma$ and let $U$ be an open subset of $M$.
The subspace $F^{\lleq \beta}\AAf^p(U)$  of $\AAf^p(U)$
is defined to be the set of forms
$\omega\in\AAf^p(U)$ such that 
for  any $\alpha_1,\dots,\alpha_p\in\Gamma$, for any
 $x\in U$ and for any $\xi_i\in F^{\ggeq\alpha_i}T_xM$, we have
 \begin{eqnarray}
   \label{eq:25M}
   \alpha_1+\cdots+\alpha_p\nlleq \beta \Rightarrow
\omega_x(\xi_1,\dots,\xi_p)=0.
 \end{eqnarray}
\end{defin}

A direct consequence of Proposition~\ref{prop:filform} is

\begin{prop}\label{prop:filform1}
  \begin{enumerate}
  \item If $\beta\lleq \gamma$ then $F^{\lleq\beta}\AAf^p\subset
    F^{\lleq\gamma}\AAf^p$.
\item Let $\beta_0\in\Gamma$ be such that $F^{\ggeq\beta_0}TM=TM$. 
If $F^{\lleq\gamma}\AAf^p\neq\{0\}$ then 
$\gamma\ggeq p\beta_0$.
\item We have $F^{\lleq 0}\AAf^p=\AAf^p$.

\item For   $\beta$ and $\gamma$ in $\Gamma$, we have
$F^{\lleq\beta}\AAf^p\wedge
    F^{\lleq\gamma}\AAf^q\subset F^{\lleq\beta+\gamma}\AAf^{p+q}$.
  \end{enumerate}
\end{prop}

\bigskip
The integrability is essential in the following result.

\begin{prop}\label{prop:filform2}
Assume that the infinitesimal filtration is $\Gamma$-integrable.
  Then for any $\beta\in\Gamma$ 
$$
d_p(F^{\lleq\beta}\AAf^p)\subset F^{\lleq\beta}\AAf^{p+1}.
$$
\end{prop}

\begin{proof}
  Let $U$ be an open subset of $M$ and let $\omega\in F^{\lleq\beta}\AAf^p(U)$.
Let $x\in U$ and let $ \xi_i\in F^{\ggeq\alpha_i}TM$ be defined in a
neighborhood of $x$ such that
$\alpha_1+\cdots+\alpha_{p+1}\nlleq \beta$.
It remains to prove that
$$
d_p(\omega)_x(\xi_1,\dots,\xi_{p+1})=0.
$$
 Cartan's formula  implies 
$$
\begin{array}{ll}
  d_p(\omega)_x(\xi_1,\dots,\xi_{p+1})=&\sum_i\pm\xi_i\cdot\omega(\xi_1,\dots,\hat{\xi_i},\dots,\xi_{p+1})\\
&+\sum_{i<j}\pm\omega_x([\xi_i,\xi_j],\xi_1,\dots,\hat{\xi_i},\dots,\hat{\xi_j},\dots,\xi_{p+1}).
\end{array}
$$
Since $[\xi_i,\xi_j]\in F^{\ggeq \alpha_i+\alpha_j}M$ and 
$$
(\alpha_i+\alpha_j)+\alpha_1+\cdots+\hat{\alpha_i}+\cdots+\hat{\alpha_j}+\cdots+\alpha_{p+1}\nlleq
\beta,
$$
the term
$\omega_x([\xi_i,\xi_j],\xi_1,\dots,\hat{\xi_i},\dots,\hat{\xi_j},\dots,\xi_{p+1})$
is zero.

Consider now a term 
\begin{eqnarray}
  \label{eq:18}
\xi_i\cdot\omega(\xi_1,\dots,\hat{\xi_i},\dots,\xi_{p+1}).  
\end{eqnarray} 
If $\alpha_i\nlleq 0$ then $\xi_i=0$ and the term~\eqref{eq:18} is zero.
Assume now that
$\alpha_i\lleq 0$.
The weight of  $\xi_1,\dots,\hat{\xi_i},\dots,\xi_{p+1}$ is
$\theta:=\sum_{j=1}^{p+1}\alpha_j - \alpha_i$.
Since $\theta + \alpha_i\nlleq\beta$ and $\alpha_i\lleq 0$,
we have $\theta\nlleq\beta$.
Since $\omega$ belongs to $F^{\lleq\beta}\AAf^p(U)$, it follows that
$\omega(\xi_1,\dots,\hat{\xi_i},\dots,\xi_{p+1})=0$.
\end{proof}

\subsection{Filtration of the cohomology}

The  $\Gamma$-filtration on $M$ induces an increasing
$\Gamma$-filtration on the cohomology. 
Indeed, Propositions~\ref{prop:filform1} and \ref{prop:filform2} show
that the De Rham complex is $\Gamma$-filtered.
Namely, we set
\begin{eqnarray}
  \label{eq:27}
  F^{\lleq \beta} \Ho^p(M,\RR):=\frac{\Ker(d_p)\cap   F^{\lleq \beta} \AAf^p(M,\RR)}
{d_{p-1}(\AAf^{p-1}(M,\RR))\cap F^{\lleq \beta} \AAf^p(M,\RR)}.
\end{eqnarray}

Propositions~\ref{prop:filform1} and \ref{prop:filform2} show the
following one.

\begin{prop}\label{prop:filcohom}
  The sets $F^{\lleq \beta} \Ho^p(M,\RR)$ are canonically identified with
  subspaces of $\Ho^p(M,\RR)$.
  \begin{enumerate}
  \item If $F^{\ggeq\beta_0}M=TM$ then $F^{\ggeq p\beta_0-\beta} \Ho^p(M,\RR)$ is a 
    $\Gamma$-filtration of $\Ho^p(M,\RR)$.
\item The filtration respects the structure  of algebra. Namely, for   
$\beta$ and  $\gamma$ in $\Gamma$, we have
$$
F^{\lleq\beta}\Ho^p(M,\RR)\,\wedge\,
    F^{\lleq\gamma}\Ho^q(M,\RR)\subset F^{\lleq\beta+\gamma}\Ho^{p+q}(M,\RR).
$$
  \end{enumerate}
\end{prop}
 
\bigskip
\begin{remark}
  The $\Gamma$-filtration is defined at the level of the de Rham
  complex and not only at the level of the cohomology. 
In particular, it induces a spectral sequence which should be study to
understand the relations between the ordinary and the Belkale-Kumar
cohomologies.  Here
we only study the Belkale-Kumar product.
\end{remark}

\bigskip
Consider now the $(\Gamma\times \ZZ)$-graded algebra associated to this
$\Gamma$-filtration of the $\ZZ$-graded (by degree) algebra $\Ho^*(M,\RR)$ by setting
\begin{eqnarray}
  \label{eq:28}
  Gr^{\beta}\Ho^p(M,\RR):=\frac{F^{\lleq\beta}\Ho^p(M,\RR)}{\sum_{\gamma\slleq\beta} F^{\lleq\gamma}\Ho^p(M,\RR)}
\end{eqnarray}
and
\begin{eqnarray}
  \label{eq:29GH}
  Gr^{\bullet}\Ho^*(M,\RR):=\bigoplus_{\beta\in\Gamma,\, p\in\NN}Gr^{\beta}\Ho^p(M,\RR).
\end{eqnarray}
Then $Gr^{\bullet}\Ho^*(M,\RR)$ is a ring graded by $\Gamma\times\ZZ$.

Now, we observe the following easy functoriality result.

\begin{lemma}
  \label{lem:funct}
Let $M$ and $N$ be two smooth manifolds endowed with integrable
infinitesimal $\Gamma$-filtrations. Let $\phi\,M\longto N$ a
smooth map such that 
$$
\forall \alpha\in \Gamma
\qquad 
T\phi(F^{\geq \alpha}TM)\subset F^{\geq \alpha}TN.
$$

Then the pullback $\phi^*\,:\, \Ho^*(N,\RR)\longto \Ho^*(M,\RR)$
respects the $\Gamma$-filtration. In particular, it induces a $\Gamma$-graded
morphism
$Gr\phi^*\,:\, Gr\Ho^*(N,\RR)\longto Gr\Ho^*(M,\RR)$.
\end{lemma}

\subsection{Cohomology with complex coefficients}

Recall that $M$ is a connected manifold.
Consider the cohomology group $\Ho^*(M,\CC)$ with complex
coefficients and consider the following complex vector bundle on $M$
$$
T^\CC M:=TM\otimes_\RR\CC.
$$ 
A {\it complex infinitesimal $\Gamma$-filtration of $M$} is a family
of complex subbundles
$$
  F^{\lleq\beta}T^\CC M\subset T^\CC M,
$$
 indexed by $\beta\in\Gamma$ satisfying the three assertions of
 Definition~\ref{def:infGammafilt}.
 A complex infinitesimal  $\Gamma$-filtration  is said to be
 {$\Gamma$-integrable} if
for any $\beta$ and $\gamma$ in $\Gamma$, we have
\begin{eqnarray}
  \label{eq:710}
  [F^{\lleq\beta}T^\CC M, F^{\lleq\gamma}T^\CC M]\subset
  F^{\lleq\beta + \gamma}T^\CC M.
\end{eqnarray}
A complex infinitesimal integrable $\Gamma$-filtration induces a
filtration of the De Rham complex and of the groups $\Ho^p(M,\CC)$.

\bigskip
\begin{exple}
  Let $M$ be an holomorphic manifold. 
Let $J$ denote  the complex structure on the tangent bundle $TM$.
Since $J^2=-{\rm Id}$, its eigenvalues acting on $TM\otimes\CC$ are
$\pm\sqrt{-1}$. Let $T^{1,0}M$ (resp. $T^{0,1}M$) denote the complex
subbundle of $TM\otimes\CC$ associated to the eigenvalue $\sqrt{-1}$
(resp. $-\sqrt{-1}$). There is  a natural $\CC$-linear 
isomorphism $\iota^{1,0}\,:\,TM\longto T^{1,0}M$.
It is well known that $T^{1,0}M$  is an integrable
  distribution in $T^\CC M$. Then we get a complex infinitesimal
  integrable $\ZZ$-filtration by setting
$$
\begin{array}{ll}
  F^{\ggeq a}T^\CC M=T^\CC M&\forall a\in\ZZ_{<0},\\
 F^{\ggeq 0}T^\CC M=T^{1,0}M,\\
 F^{\ggeq a}T^\CC M=\underline{0},&\forall a\in\ZZ_{>0}.\\
\end{array}
$$
The $\ZZ$-filtration of $\Ho^p(M,\CC)$ is called the Hodge filtration of
$M$ (see for example \cite{Voisin:Hodge1}).
\end{exple}

\subsection{The case of a smooth complex variety}
\label{sec:defclassY}

Let $M$ be a smooth complex irreducible variety
 endowed with an algebraic
$\Gamma$-filtration. 
Assume that this filtration is integrable and
comes from a decomposition (recall the definition from Section~\ref{sec:manifolds}). 
Set $\tilde\Gamma:=\Gamma\times\ZZ$ endowed with the order
$(\beta,n)\ggeq(\gamma,m)$ if and only if $\beta\ggeq\gamma$ and
$n\geq m$.

Define a complex $\tilde\Gamma$-filtration on $T^\CC M$ by  setting
for any $\beta\in\Gamma$,
 $$
\begin{array}{ll}
  F^{\ggeq(\beta,a)}T^\CC M=T^\CC M&\forall a\in\ZZ_{<0},\\
 F^{\ggeq (\beta,0)}T^\CC M=\iota^{1,0}(F^{\ggeq\beta}TM),\\
 F^{\ggeq(\beta,a)}T^\CC M=\underline{0},&\forall a\in\ZZ_{>0}.\\
\end{array}
$$

\bigskip
\noindent{\bf Integration along subvarieties.}
Let $N$ be an irreducible subvariety of $M$.
Denote by $n$ the dimension of $M$ and by $d$ that of $N$. 
By Lemma~\ref{lem:decdim}, the dimension vector $(fd^\beta(T_xN))_{\beta\in\Gamma}$ does
not depend on $x\in N$ general. 
This general value of the dimension vector is by definition the {\it
  $f$-dimension vector of $N$} and is denoted by $fd^\beta(N)$.
For any $x$ in $N$, the $\Gamma$-filtration of $T_xN$ comes from a
decomposition by Lemma~\ref{lem:decsous}.
In particular, Lemma~\ref{lem:decdim}  shows that the $g$-dimensional
vector of $T_xN$ does not depend on $x$ in $N$ general.
This remark allows to define the $g$-dimension vector of $N$. 
Then the  weight $\rho(N)\in\Gamma$ of $N$ is defined by the formula
 
\begin{eqnarray}
  \label{eq:31}
 \rho(N)=\sum_{\beta\in\Gamma}gd^\beta(N)\beta.
\end{eqnarray}
Consider the extended notions to $\tilde\Gamma$:
 $\widetilde{gd}^{(\beta,0)}(N)=gd^{(\beta,0)}(T_xN\otimes\CC)=
gd^{\beta}(N)$, $\widetilde{gd}^{(0,-1)}(N)=d$ and 
$\widetilde{gd}^{(\beta,a)}(N)=0$ otherwise. Note that $\tilde\rho(N)=(\rho(N),-d)$.

Consider now the linear map
$$
\begin{array}{ccc}
  \AAf^{2d}(M,\CC)&\longto&\CC\\
\omega&\longmapsto&\int_N\omega_{|N}.
\end{array}
$$
The following lemma relies the filtration and the integration.

\begin{lemma}
  \label{lem:intF}
Let $\beta\in\Gamma$ and let $e\in \ZZ$ such that 
$(\beta,e)\nggeq\tilde\rho(N)$.
If $\omega\in F^{\lleq(\beta,e)}\AAf^{2d}(M,\CC)$ then
$$\int_N\omega_{|N}=0.
$$
\end{lemma}

\begin{proof}
Let $x\in N$ be a general point. 
By Lemma~\ref{lem:decsous}, the $\Gamma$-filtration on $T_xN$ comes
from a decomposition. Then there exists
a basis $(\xi_1,\dots,\xi_d)$
of $T_xN$ such that for any $\beta\in\Gamma$, the set of $\xi_i$ which
belong to $F^{\ggeq\beta}T_xN$ spans    $F^{\ggeq\beta}T_xN$.
Such a basis exists since by Lemma~\ref{lem:decsous}, the
$\Gamma$-filtration
on $T_xN$ comes from a decomposition.
Let $\alpha_i$ be the maximal element of $\Gamma$ such that $\xi_i$
belongs to $F^{\ggeq\alpha_i}T_xN$.

Consider the basis
$(\iota^{(1,0)}(\xi_1),\dots,\iota^{(1,0)}(\xi_d),\iota^{(0,1)}(\xi_1),\dots,\iota^{(0,1)}(\xi_d))$
of $T_xN\otimes \CC$. Since $x$ is any general point on $N$, it is
sufficient to prove that 
$$\omega(\iota^{(1,0)}(\xi_1),\dots,\iota^{(1,0)}(\xi_d),\iota^{(0,1)}(\xi_1),\dots,\iota^{(0,1)}(\xi_d))=0.$$
But $\iota^{(1,0)}(\xi_i)\in F^{\ggeq(\alpha_i,0)} T^\CC N$ and
$\iota^{(0,1)}(\xi_i)\in F^{\ggeq(0,-1)} T^\CC N$.
Hence the weight of
$(\iota^{(1,0)}(\xi_1),\dots,\iota^{(1,0)}(\xi_d),\iota^{(0,1)}(\xi_1),\dots,\iota^{(0,1)}(\xi_d))$
is $\sum_{i=1}^d(\alpha_i,0)+d (0,-1)=\tilde\rho(N)$. The lemma follows.
\end{proof}

\bigskip
The restriction of the map $\omega\longmapsto \int_N\omega_{|N}$ to
the closed $2d$-forms is zero on the exact forms and induces a linear
map
$$
\int_N\,:\,\Ho^{2d}(M,\CC)\longto \CC.
$$
Consider now the restriction of this map to $F^{\lleq
  \tilde\rho(N)}\Ho^{2d}(M,\CC)$. By Lemma~\ref{lem:intF}, this
restriction induces a linear map
$$
\int_N\,:\,\G^{\tilde\rho(N)} \Ho^{2d}(M,\CC)\longto \CC.
$$

\bigskip
\noindent{\bf Poincar\'e pairing.}
Assume that $M$ is compact and recall that it is orientable since
it is holomorphic. 
Let $p$ be an integer such that $0\leq p\leq 2d$.
The integration allows to define a paring
\begin{eqnarray}
  \label{eq:745}
\begin{array}{ccc}
  \Ho^p(M,\CC)\times \Ho^{2d-p}(M,\CC)&\longto&\CC\\
([\omega_1],[\omega_2])&\longmapsto&\int_M\omega_1\wedge\omega_2.
\end{array}
\end{eqnarray}
By Poincar\'e duality, this bilinear form is non degenerated. In
particular, $\Ho^p(M,\CC)$ and  $\Ho^{2d-p}(M,\CC)$ have the same dimension.
 
Let $\tilde\alpha\in\tilde\Gamma$. Consider the following restriction
of the bilinear form~\eqref{eq:745}:
\begin{eqnarray}
  \label{eq:739}
  \begin{array}{ccc}
  F^{\lleq\tilde \alpha}\Ho^p(M,\CC)\times F^{\lleq \tilde\rho(M)-\tilde\alpha}\Ho^{2d-p}(M,\CC)&\longto&\CC\\
([\omega_1],[\omega_2])&\longmapsto&\int_M\omega_1\wedge\omega_2.
\end{array}
\end{eqnarray}
Since
$$
\begin{array}{l}
\tilde\alpha\ggeq \tilde\rho(M)\Rightarrow F^{\lleq
  \tilde\alpha}\AAf^{2d}(M)=\AAf^{2d}(M),\mbox{ and}\\
\tilde\alpha\nggeq \tilde\rho(M)\Rightarrow F^{\lleq \tilde\alpha}\AAf^{2d}(M)=0,
\end{array}
$$
Lemma~\ref{lem:intF} shows that 
\begin{eqnarray}
  \label{eq:737}
\tilde\alpha+\tilde\beta \nggeq \tilde\rho(M) \Rightarrow F^{\lleq
  \tilde \alpha} \Ho^p(M,\CC)\wedge
 F^{\lleq \tilde\beta} \Ho^{2d-p}(M,\CC)=\{0\}.
\end{eqnarray}
In particular,  the pairing~\eqref{eq:739} passes to the quotient and
induces a pairing
\begin{eqnarray}
  \label{eq:746}
  \begin{array}{ccc}
  \G^{\tilde \alpha}\Ho^p(M,\CC)\times \G^{\tilde\rho(M)-\tilde\alpha}\Ho^{2d-p}(M,\CC)&\longto&\CC\\
([\omega_1],[\omega_2])&\longmapsto&\int_M\omega_1\wedge\omega_2.
\end{array}
\end{eqnarray}

\begin{defin}
  The $\tilde\Gamma$-filtration of $\Ho^*(M,\CC)$ is
 said to be {\it compatible with Poincar\'e duality} if for any integer
 $0\leq p\leq 2d$ and
  for any $\tilde\alpha\in\tilde\Gamma$, the pairing~\eqref{eq:746}
is non degenerate.
\end{defin}

\bigskip
\begin{lemma}
\label{lem:PDdim}
  The $\tilde\Gamma$-filtration of $\Ho^*(M,\CC)$ is
 compatible with Poincar\'e duality if and only if 
for any nonnegative integer $p$ and any $\tilde\alpha\in\tilde\Gamma$, we have
\begin{equation}
  \label{eq:14}
  \dim(\G^{\tilde \alpha}\Ho^p(M,\CC))=\dim(\G^{\tilde\rho(M)-\tilde\alpha}\Ho^{2d-p}(M,\CC))
\end{equation}
\end{lemma}

\begin{proof}
If the $\tilde\Gamma$-filtration of $\Ho^*(M,\CC)$ is
 compatible with Poincar\'e duality we obviously have the equalities
 of dimensions.

Assume now that \eqref{eq:14} hold. 
In a basis adapted to the filtration, implication~\eqref{eq:737}
implies that the matrix $A$ of the pairing~\eqref{eq:745} is upper triangular. 
Moreover, the matrices (in the induced basis) of the
pairings~\eqref{eq:746} are the diagonal blocs of $A$.
 But equalities~\eqref{eq:14} imply that these blocs are square.
Since $A$ is invertible, it follows that any bloc is invertible. 
\end{proof}

\bigskip
\begin{defin}
  Let $N$ be an irreducible subvariety of a compact smooth irreducible
  complex variety $M$ endowed with an integrable infinitesimal
  $\Gamma$-filtration coming from a decomposition.
Assume that the $\tilde\Gamma$-filtration is   compatible with Poincar\'e duality.
Define
$[N]_\bkprod\in\G^{\tilde\rho(M)-\tilde\rho(N)}\Ho^{2(n-d)}(M,\CC))$ to
satisfy 
the following formula
\begin{eqnarray}
  \label{eq:33}
  \int_N[\omega]=\int_M [N]_\bkprod\wedge[\omega],
\end{eqnarray}
for any $[\omega]\in\G^{\tilde\rho(N)}\Ho^{2d}(M,\CC)$.
\end{defin}

On can refer to Proposition~\ref{prop:BKclass} for a more concerte
characterization of $[N]_\bkprod$ and in particular its relation with
$[N]$, in the case when $M=G/P$.

\section{Isomorphism with the Belkale-Kumar product}

\subsection{The Belkale-Kumar product}
\label{sec:BKprod}

In this section, we recall the Belkale-Kumar notion of Levi-movability (see~\cite{BK}).\\

The cycle class of the Schubert variety $X_w$ in $\Ho^*(G/P,\CC)$ is denoted by $\sigma_w$
and it is called a Schubert class. 
The degree of $\sigma_w$ is $2(\dim G/P-l(w)$, where
$l(w)=\sharp\Phi(w)$ is the length of $w$.
The Schubert classes form a basis of the cohomology of $G/P$:
\begin{eqnarray}
  \label{eq:35}
  \Ho^*(G/P,\CC)=\bigoplus_{w\in W^P}\CC\sigma_w.
\end{eqnarray}

The Poincar\'e dual of  $\sigma_w$ is denoted by $\sigma_w^\vee$.
Note that $\sigma_e$ is the class of the point.
Let $\sigma_u,\,\sigma_v,\,\sigma_w$ be three Schubert classes (with
$u,v,w\in W^P$). 
If there exists an integer $d$ such that
$\sigma_u.\sigma_v.\sigma_w=d\sigma_e$ then we set $c_{uvw}=d$; we set $c_{uvw}=0$ otherwise. 
These coefficients are the (symmetrized) structure coefficients 
of the cup product on 
${\rm H}^*(G/P,\CC)$ in the Schubert basis in the following sense:
$$
\sigma_u.\sigma_v=\sum_{w\in W^P} c_{uvw}\sigma_w^\vee
$$
and $c_{uvw}=c_{vuw}=c_{uwv}$.\\

Consider the tangent space $T_u$ of the orbit
 $u^{-1}BuP/P$ at the point $P/P$; and, similarly consider $T_v$ and $T_w$. 
Using the transversality theorem 
of Kleiman, Belkale and Kumar showed in \cite[Proposition~2]{BK} the 
following important lemma.

\begin{lemma}\label{lem:fondBK}
  The coefficient $c_{uvw}$ is nonzero if and only if there exist $p_u,p_v,p_w\in P$ such that 
the natural map
$$
T_P(G/P)\longto \frac{T_P(G/P)}{p_uT_u}\oplus \frac{T_P(G/P)}{p_vT_v}\oplus \frac{T_P(G/P)}{p_wT_w}
$$
is an isomorphism.
\end{lemma}

Then Belkale-Kumar defined Levi-movability.\\

\begin{defin}
The triple $(\sigma_u,\,\sigma_v,\,\sigma_w)$ is said to be {\it Levi-movable} if there exist
 $l_u,l_v,l_w\in L$ such that the natural map
$$
T_P(G/P)\longto \frac{T_P(G/P)}{l_uT_u}\oplus \frac{T_P(G/P)}{l_vT_v}\oplus \frac{T_P(G/P)}{l_wT_w}
$$
is an isomorphism.
\end{defin}

\bigskip
Belkale-Kumar  set
$$
c_{uvw}^\kbprod=\left\{
  \begin{array}{ll}
c_{uvw}&{\rm \ if\ }  (\sigma_u,\,\sigma_v,\,\sigma_w) {\rm\ is\ Levi-movable;}\\
    0&{\rm\ otherwise.}
  \end{array}
\right .
$$
They defined on the group ${\rm H}^*(G/P,\CC)$ a bilinear product $\kbprod$ by the formula
$$
\sigma_u\kbprod\sigma_v=\sum_{w\in W^P} c_{uvw}^\kbprod\sigma_w^\vee.
$$

\begin{theo}[Belkale-Kumar 2006]
\label{th:BK}
  The product $\kbprod$ is commutative, associative and satisfies Poincar\'e duality.
\end{theo}


\cite[Proposition~2.4]{RR} gives  an equivalent characterization of Levi-movability. 
It can be formulated as follows.

\begin{prop}
\label{Prop_Dim}
Let $u,\,v,\,w\in W^P$ such that
$c_{uvw}\neq 0$.
Then $(\sigma_u,\,\sigma_v,\,\sigma_w) $ is Levi-movable if and only if

$$2gd(G/P)=gd(X_u)+gd(X_v) +gd(X_w).$$
\end{prop}

\subsection{The statements}

The first aim of this section is to prove (see
Section~\ref{sec:proofs}) the following result of
compatibility between the basis of Schubert classes and the
$\tilde\Gamma$-filtration on $\Ho^*(G/P,\CC)$.

\begin{prop}
  \label{prop:filtSchub}
For any $\tilde\beta\in\tilde\Gamma$ and for any integer $p$, the linear
subspace $F^{\lleq\tilde\beta}\Ho^p(G/P,\CC)$ is spanned by the
Schubert classes it contains.

More precisely, $F^{\lleq\tilde\beta}\Ho^p(G/P,\CC)$ is spanned by the
Schubert classes $\sigma_{w^\vee}$ where $w\in W^P$ satisfies $(\rho(X_w),-l(w))\lleq\tilde\beta$. 
\end{prop}

For any $w\in W^P$,  denote by $\overline{\sigma_{w^\vee}}$ the
class of $\sigma_{w^\vee}\in F^{ \lleq(\rho(X_w),-l(w))}\Ho^{l(w)}(G/P,\CC)$
in $\G^{ (\rho(X_w),-l(w))}\Ho^{l(w)}(G/P,\CC)$.
Proposition~\ref{prop:filtSchub} implies that
$(\overline{\sigma_{w^\vee}})_{w\in W^P}$ is a basis of $\G \Ho^*(G/P,\CC)$.
Consider now the obvious linear isomorphism
$$
\begin{array}{ccccc}
  \Psi\,:&\Ho^*(G/P,\CC)&\longto&\G \Ho^*(G/P,\CC)\\
&\sigma_{w^\vee}&\longmapsto&\overline{\sigma_{w^\vee}}&{\rm\ for\ any\ }w\in W^P.
\end{array}
$$

\begin{theo}\label{th:iso}
  The linear isomorphism $\Psi$ from the algebra $(\Ho^*(G/P,\CC),\bkprod)$ onto
  the algebra $\G \Ho^*(G/P,\CC)$ is an isomorphism of algebras.
\end{theo}

The theorem is proved in Section~\ref{sec:proofth} after some
preparation. 
The first consequence concerns Poincar\'e duality (see Section~\ref{sec:proofs}).

\begin{coro}
 \label{cor:PoincareGP}
The $(X(Z)\times\ZZ)$-filtration of $\Ho^*(G/P,\CC)$ is compatible with
Poincar\'e duality.  
\end{coro}

This corollary allows  to define the graded Schubert classes by setting,
for any $w\in W^P$,
\begin{eqnarray}
  \label{eq:34}
  \sigma_w^\bkprod:=[X_w]_\bkprod.
\end{eqnarray}

Finally, we get, by applying Proposition~\ref{prop:BKclass} to $Y=X_w$, the following result of compatibility.

\begin{lemma}
  For any $w\in W^P$, we have
$$
\Psi(\sigma_w)=\sigma_w^\bkprod.
$$
\end{lemma}

\subsection{An upper bound for $\dim(F^{\lleq\tilde\alpha}\Ho^p(G/P,\CC))$}

For any $w\in W$, as a consequence of the relation $\Phi^-=
(\Phi^-\cap w^{-1}\Phi^+)\cup (\Phi^-\cap w^{-1}\Phi^-)$,
we have (see \cite[1.3.22.3]{Kumar:KacMoody})
\begin{eqnarray}
  \label{eq:236a}
  \sum_{\alpha\in\Phi^-\cap w^{-1}\Phi^+}\alpha=w^{-1}\rho-\rho.
\end{eqnarray}
Assume that $w\in W^P$.
Since $P/P$ is 
$X(T)$-regular and  $T$ acts on $T_{P/P}w^{-1}X_w$ without
multiplicities  and with
weights $\Phi^-\cap w^{-1}\Phi^+$, we have
\begin{eqnarray}
  \label{eq:239}
  \rho(X_w)=\rho(w^{-1}X_w)=\left(\sum_{\alpha\in\Phi^-\cap
      w^{-1}\Phi^+}\alpha\right)_{|Z}=\left(
w^{-1}\rho-\rho
\right)_{|Z}.
\end{eqnarray}
In particular
\begin{eqnarray}
  \label{eq:237}
  \rho(G/P)=2\left(
\rho_L-\rho
\right)_{|Z}=-2\rho_{|Z},
\end{eqnarray}
since $\rho_L$ is trivial on $Z$.
Hence
\begin{eqnarray}
  \label{eq:238}
  \rho(G/P)-\rho(X_w)=\left(
-\rho-w^{-1}\rho
\right)_{|Z}.
\end{eqnarray}

\begin{lemma}
\label{lem:calcul}
For any $w\in W^P$, we have
$$\rho(G/P)-\rho(X_w) =\rho(X_{w^\vee}).$$
\end{lemma}

\begin{proof}
Remark that
$$
((w^\vee)^{-1}\rho)_{|Z}=((w_0^Pw^{-1}w_0\rho)_{|Z}=-w_0^P(w^{-1}\rho)_{|Z}=-(w^{-1}\rho)_{|Z},
$$
since $w_0^P$ belongs to $L$ and acts trivially on $Z$.
The lemma follows.
\end{proof}

\begin{lemma}\label{lem:dimF}
Let $n$ denote the dimension of $G/P$.
  The dimension of $F^{\lleq \beta}\Ho^{2(n-d)}(G/P,\CC)$ is less or equal
  to the number of $w\in W^P$ such that $\rho(G/P)-\rho(X_w)\lleq\beta$
  and $l(w)=d$.
\end{lemma}

\begin{proof}
For each   $w\in W^P$ such that $\rho(G/P)-\rho(X_w)\nlleq\beta$
  and $l(w)=d$, consider the linear form
$$
\int_{X_{w^\vee}}\,:\Ho^{2(n-d)}(G/P,\CC)\longto\CC.
$$
By Lemmas~\ref{lem:calcul} and \ref{lem:intF}, this linear form is
zero on $F^{\lleq \beta}\Ho^{2(n-d)}(G/P,\CC)$.
But by Poincar\'e duality these linear forms are linearly independent. 
This implies that the codimension of 
$F^{\lleq \beta}\Ho^{2(n-d)}(G/P,\CC)$ in $\Ho^{2(n-d)}(G/P,\CC)$ is at least
the number of $w\in W^P$ such that $\rho(G/P)-\rho(X_w)\nlleq\beta$
  and $l(w)=d$. The lemma follows.
\end{proof}

\subsection{Kostant's harmonic forms}

\subsubsection{The role of Kostant's harmonic forms in this paper}

Let $w$ in $W^P$.
In 1963, B.~Kostant constructed an explicit $\CC$-valued closed
differential form $\omega_w$ on $G/P$ such that the associated cohomology class
$[\omega_w]$ is equal to $\sigma_w$ up to a scalar multiplication.
Kostant's form $\omega_w$ is used here to localize the
Schubert class relatively to the filtration.

\begin{lemma}
 \label{lem:Schubfil}
The Schubert class $\sigma_{w^\vee}$ belongs to $F^{\lleq(\rho(X_w),-l(w)) }\Ho^{l(w)}(G/P,\CC)$.
\end{lemma}

Before proving Lemma~\ref{lem:Schubfil} in Section~\ref{sec:proofs}, we recall 
Kostant's construction.

\subsubsection{Restriction to $K$-invariant forms}

Let $K$ be a maximal compact subgroup of $G$ such that $T\cap K$ is a
maximal torus of $K$. 
Then $K$ is a connected compact Lie group. 

Consider the subcomplex of  $K$-invariant forms:
$$
d_p\,:\,\AAf^p(G/P,\CC)^K\longto\AAf^{p+1}(G/P,\CC)^K,
$$
and its cohomology $\Ho^*_{DR}(G/P,\CC)^K$.
The identity
$d_{p-1}(\AAf^{p-1}(G/P,\CC)^K)=d_{p-1}(\AAf^{p-1}(G/P,\CC))\cap
\AAf^p(G/P,\CC)^K$
allows  to define a morphism 
$$
\Ho^*_{DR}(G/P,\CC)^K\longto \Ho^*_{DR}(G/P,\CC),
$$
which is an isomorphism.

Since $K$ acts transitively on $G/P$, 
the restriction map to the tangent space at $P/P$ provides a linear isomorphism
\begin{eqnarray}
  \label{eq:137}
  \AAf^p(G/P,\CC)^K\longto\left(\bigwedge^p \Hom_\RR(\lieg/\lp,\CC)\right)^{K\cap L}.
\end{eqnarray}
Let $\lk$ denote the Lie algebra of $K$.
This compact form $\lk$ determines a real structure $\square^*$ on $\lieg$.
More precisely, $\square^*$ is a $\CC$-antilinear endomorphism of
$\lieg$  such that $\lk$ is the set of $\xi\in\lieg$ such that $\xi^*=-\xi$.

Consider now the complex dual $(\lieg/\liel)^*$ of the complex vector
space $\lieg/\liel$.
Since $\liel$ is stable by $\square^*$, $\lieg/\liel$ is endowed with a real
structure still denoted by $\square^*$.
Then $(\lieg/\liel)^*$ is also endowed with a real structure by setting
$\varphi^*=\overline{\varphi(\square^*)}$, for any $\varphi\in (\lieg/\liel)^*$.
Define a morphism
$$
\begin{array}{cccc}
  \theta\;:&\Hom_\RR(\lieg/\lp,\CC)&\longto&(\lieg/\liel)^*\\
&\varphi+\psi&\longmapsto&\varphi+\psi(\square^*),
\end{array}
$$
where $\varphi$ is $\CC$-linear and $\psi$ is $\CC$-antilinear.
One checks that $\theta$ is a $\CC$-linear isomorphism and that it commutes
with the real structure and the actions of $K\cap L$.
Note that $L$ also acts on $(\lieg/\liel)$. Since $K\cap L$ is Zariski
dense in $L$, we have 
\begin{eqnarray}
  \label{eq:135}
  \left(\bigwedge\nolimits^p (\lieg/\liel)^*\right)^{K\cap L}=\left(\bigwedge\nolimits^p (\lieg/\liel)^*\right)^{L}.
\end{eqnarray}
Finally we get an isomorphism
\begin{eqnarray}
  \label{eq:138}
  \AAf^p(G/P,\CC)^K\longto\left(\bigwedge\nolimits^p (\lieg/\liel)^*\right)^{L}.
\end{eqnarray}

\subsubsection{The Lie algebra $\lr$}

Let $\lu$ and $\lu^-$ be the algebras of the unipotent radicals of $P$
and its opposite parabolic subgroup $P^-$. 
Consider the sum
\begin{eqnarray}
  \label{eq:125}
  \lr=\lu^-\oplus\lu
\end{eqnarray}
endowed  with a Lie algebra structure $[\cdot,\cdot]_\lr$ defined
 by keeping the brackets on $\lu^-$ and $\lu$ unchanged and by setting $[\lu^-,\lu]_\lr=0$.
The $L$-equivariant linear isomorphism
$\lr\simeq\lieg/\liel$ and its transpose
$(\lieg/\liel)^*\simeq \lr^*$
induce isomorphisms
\begin{eqnarray}
  \label{eq:139}
   \AAf^\bullet (G/P,\CC)^K\simeq\left(
\Hom_\RR(\lieg/\lp,\CC)
\right)^{L}\simeq
\left(\bigwedge\nolimits^\bullet\lr^*\right)^L\simeq  
\left(
\bigwedge\nolimits^\bullet
(\lu^-)^*\otimes \bigwedge\nolimits^\bullet\lu^*
\right)^L.
\end{eqnarray}
The term $\bigwedge^\bullet
(\lu^-)^*$ corresponds to holomorphic forms on $G/P$ and the term $\bigwedge^\bullet
\lu^*$ corresponds to antiholomorphic forms.

Combining $\square^*$ and the Killing form  $(\cdot,\cdot)$  one obtains an
Hermitian form $\{\cdot,\cdot\}$ on $\lieg$.
Explicitly,
$$
\{\xi,\eta\}=-(\xi,\eta^*),
$$
for any $\xi,\,\eta\in\lieg$.
Denote by $\{\cdot,\cdot\}_\lr$ its restriction to $\lr$.
The decomposition $\lu^-\oplus\lu=\lr$ is orthogonal
for $\{\cdot,\cdot\}_\lr$.
Consider now the graded exterior algebra
$\wedge^\bullet\lr^*=\oplus_{p}\wedge^p\lr^*$ and extend the bilinear form $\{\cdot,\cdot\}_\lr$  on
$\wedge^\bullet\lr^*$.
The decomposition $\lr=\lu^-\oplus\lu$ induces a $\NN^2$-grading 
$\wedge^\bullet\lr^*=\oplus_{(p,q)\in\NN^2}\wedge^{p,q}\lr^*$ by setting
$$
\wedge^{p,q}\lr^*=\wedge^p(\lu^-)^*\otimes
\wedge^q(\lu)^*.
$$
Moreover, the sum $\oplus_{(p,q)\in\NN^2}\wedge^{p,q}\lr^*$ is
orthogonal for $\{\cdot,\cdot\}_\lr$ .

Let $b\in\End(\wedge^\bullet\lr^*)$ be the Chevalley-Eilenberg
coboundary operator of the Lie algebra $\lr$. It has degree $+1$, more precisely
$$
b(\wedge^{p,q}\lr^*)\subset \wedge^{p+1,q}\lr^*\oplus \wedge^{p,q+1}\lr^*.
$$
Set $b=b^{1,0}+b^{0,1}$ according to this decomposition. 
Let  $\partial\in\End(\wedge^\bullet\lr)$ denote the Chevalley-Eilenberg boundary operator.
Using the Killing form, we identify $\lr$ and $\lr^*$ and transport
$\partial$ to an operation 
$\partial^*\in\End(\wedge^\bullet\lr^*)$ of degree $-1$. 
Decompose $\partial^*=\partial^{-1,0}+\partial^{0,-1}$ according to the
decomposition $\wedge^{p-1,q}\lr^*\oplus \wedge^{p,q-1}\lr^*$.
Set 
\begin{eqnarray}
  \label{eq:836}
  \KoL=\partial^*\circ b+b\circ\partial^*.
\end{eqnarray}
 \cite[Proposition~4.2]{Kostant:harmform2} gives an alternative expression of $\KoL$:
\begin{eqnarray}
  \label{eq:837}
   \KoL=\frac 1 2 (\partial^{0,-1}\circ b^{0,1}+b^{0,1}\circ\partial^{0,-1}).
\end{eqnarray}

\subsubsection{The $(X(Z)\times \ZZ)$-filtration of $\wedge^\bullet\lr^*$}

Consider the action of $Z\times \CC^*$ on $\lr$ given by 
\begin{eqnarray}
  \label{eq:214}
  (z,\tau).(\xi^-+\xi)=(\tau z\xi^-,\xi),\qquad
\forall z\in Z,\,\tau\in\CC^*,\,\xi^-\in\lu^-,\,\xi\in\lu.
\end{eqnarray}
Then  the group $Z\times\CC^*$ acts on $\wedge^\bullet\lr^*$ and induces a
$\tilde\Gamma$-decomposition
\begin{eqnarray}
  \label{eq:75}
  \wedge^\bullet\lr^*=\bigoplus_{\tilde\beta\in X(Z)\times\ZZ} (\wedge^\bullet\lr^*)_{\tilde\beta}.
\end{eqnarray}
Note that the weights of $Z$ acting on $(\lu^-)^*$ are the weights of $Z$
acting on $\lu$; in particular, they are positive for the order
$\ggeq$.
As a consequence, we have
\begin{eqnarray}
  \label{eq:236}
(\wedge^\bullet\lr^*)_{\tilde\beta}\neq\{0\}\,\Rightarrow\,\tilde\beta\ggeq
0.
\end{eqnarray}
Set
\begin{eqnarray}
  \label{eq:2399}
  F^{\lleq\tilde\beta}(\wedge^\bullet\lr^*)=\oplus_{\tilde\alpha\lleq\tilde\beta}(\wedge^\bullet\lr^*)_{\tilde\alpha}.
\end{eqnarray}
Consider now, like in the formula~\eqref{eq:139}, the diagonal action
of $L$ on $\lr$:
$$
l.(\xi^-+\xi)=l\xi^-+l\xi,\qquad \forall l\in L,\,\xi^-\in\lu^-,\,\xi\in\lu.
$$
Since $Z$ is contained in the center of $L$; the action~\eqref{eq:214}
of $Z\times \CC^*$ and the above action of $L$ commute. 
In particular the decomposition~\eqref{eq:75} is
$L$-stable.
Set $C=(\wedge^\bullet\lr^*)^L$ and $C_{\tilde\beta}=C\cap
(\wedge^\bullet\lr^*)_{\tilde\beta}$. The $(Z\times\CC^*)$-module $C$
decomposes as follows
 \begin{eqnarray}
   \label{eq:140}
   C:=\bigoplus_{\tilde\beta\in\tilde\Gamma} C_{\tilde\beta} .
 \end{eqnarray}
The associated filtration of $C$ is: 
$$
 F^{\lleq\tilde\beta}\,C
 =F^{\lleq\tilde\beta}(\wedge^\bullet\lr^*)\cap C.
$$

\subsubsection{Action of $L$ on $\wedge^\bullet(\lu^-)^*$}

We now recall results of Kostant in \cite{Kostant:harmform1} on the
action of $T$ on $\wedge^\bullet(\lu^-)^*$.

\begin{theo}\label{th:Ko1}
  \begin{enumerate}
  \item 
The set of vertices of the convex hull of the weights of $T$ acting on 
$\bigwedge^\bullet(\lu^-)^*$ is the set of $\rho-w^{-1}\rho$ where $w\in W^P$.

These weights are multiplicity free and the eigenline corresponding to
$\rho-w^{-1}\rho$ is generated by 
$$
\phi_w:=\phi_{\alpha_1}\wedge\dots\wedge \phi_{\alpha_p},
$$
where $\{\alpha_1,\dots,\alpha_p\}=\Phi^+\cap
w^{-1}\Phi^-$; and $\phi_{\alpha_i}\in(\lu^-)^*$ is a  vector of weight
$\alpha_i$.
\item For any $w\in W^P$, the vector $\phi_w$ is an highest weight
  vector for $L$.
Denote by $M_w$ the simple $L$-module generated by $\phi_w$.
\end{enumerate}
\end{theo}

\subsubsection{A first differential form}

We are now ready to define a first $K$-invariant differential form on $G/P$. Set
\begin{eqnarray}
  \label{eq:143}
  h_w={\rm Id}_w\in M_w\otimes M_w^*\subset \left(\wedge^p(\lu^-)^*\otimes \wedge^p\lu^*\right)^L,
\end{eqnarray}
where $p$ is the length of $w$, that is s the codimension of $X_{w^\vee}$.
Since $Z$ is central in $L$, $Z$ acts with weight
$(\rho-w^{-1}\rho)_{|Z}$ on $M_w$. In particular, 
\begin{eqnarray}
  \label{eq:144}
  h_w\in C_{((\rho-w^{-1}\rho)_{|Z},-p)}.
\end{eqnarray}

If $G/P$ is cominuscule then $h_w$ corresponds by the
isomorphism~\eqref{eq:139} to the wanted closed differential
form representing $\sigma_w$. In general, more work is useful.

\subsubsection{An Hermitian product on $\lr$}

Recall that 
the Hermitian product $\{\cdot,\cdot\}_\lr$ on $\lr$ 
induces Hermitian inner products on $\wedge^\bullet\lr$ and
$\wedge^\bullet\lr^*$ still denoted by $\{\cdot,\cdot\}_\lr$.

\begin{lemma}
 \label{lem:decomporth}
For any nonnegative integer $p$, the
$(X(Z)\times\ZZ)$-decomposition~\eqref{eq:140} is $\{\cdot,\cdot\}_\lr$-orthogonal.  
\end{lemma}

\begin{proof}
It is sufficient to prove that the decomposition
\begin{eqnarray}
  \label{eq:1388}
  \lr=\lu\oplus\bigoplus_{\alpha\in X(Z)}\lu^-_\alpha
\end{eqnarray}
is $\{\cdot,\cdot\}_\lr$-orthogonal.  
Since $\lu^*=\lu^-$ and the Killing form vanishes on $\lu^-$, $\lu$
and $\lu^-$ are $\{\cdot,\cdot\}_\lr$-orthogonal.  
Let now fix $\xi\in \lu^-_\beta$ and $\eta\in\lu^-_{\beta'}$ with
$\beta\neq\beta'\in X(Z)$.
  Consider the adjoint action of $Z$ on $\lieg$, the induced one on
  $\End(\lieg)$ and the corresponding decomposition
$$
\End(\lieg)=\oplus_{\alpha\in X(Z)}\End(\lieg)_\alpha.
$$
Note that for any $A\in\End(\lieg)_\alpha$ with $\alpha\neq 0$, we have
$\tr(A)=0$. 
The endomorphism $\Ad(\eta^*)$ belongs to $\End(\lieg)_{-\beta'}$.
It follows that $\Ad(\eta^*)\circ\Ad(\xi)$ belongs to
$\End(\lieg)_{\beta-\beta'}$
and that $\{\xi,\eta\}=-(\xi,\eta^*)=-\tr(\Ad(\eta^*)\circ\Ad(\xi))=0$.
\end{proof}

\subsubsection{Operators on $\wedge^\bullet(\lr^*)$}

Recall, from the formula~\eqref{eq:836}, the definition of the operator
$\KoL\in\End(\wedge^\bullet\lr^*)$.

\begin{lemma}
 \label{lem:opLstab}
The operator $\KoL$ stabilizes $C_{(\alpha,p)}$ for any 
integer $p$ and  any $\alpha\in X(Z)$.
\end{lemma}

\begin{proof}
  By \cite[Proposition~3.4]{Kostant:harmform2}, $b^{0,1}(C_{(\alpha,p)})$
is contained in $C_{(\alpha,p+1)}$.
By \cite[formula~3.5.3]{Kostant:harmform2}, $\partial^{0,-1}(C_{(\alpha,p+1)})$
is contained in $C_{(\alpha,p)}$.
We deduce that $(\partial^{0,-1}\circ b^{0,1})(C_{(\alpha,p)})$ is
contained in $C_{(\alpha,p)}$. Similarly, 
 $(b^{0,1}\circ\partial^{0,-1}) (C_{(\alpha,p)})$ is
contained in $C_{(\alpha,p)}$. 
We conclude using the formula~\eqref{eq:837}.
\end{proof}

\bigskip
Note that $\KoL$ is an Hermitian operator. In particular, we have a 
$\{\cdot,\cdot\}_\lr$-orthogonal decomposition
$\Ker\KoL\oplus\Im\KoL=\wedge^\bullet\lr^*$.
Consider the quasiinverse $\KoL_0$ of $\KoL$:
$\KoL_0$  is the Hermitian
endomorphism of $ \wedge^\bullet\lr^*$ such that $\Ker\KoL_0=\Ker\KoL$
and ${\KoL_0}_{|\Im\KoL}=(\Li_{|\Im\KoL})^{-1}$.

\bigskip
Let $\pi\,:\,\lr\longto\End(\wedge^\bullet\lr^*)$ be induced by the coadjoint
action.
Let $f_i$ be eigenvectors in $\lu^-$ for the action of $Z$ that form a
basis of $\lu^-$. 
Let $g_j$ be the basis of $\lu$
defined by the conditions $(f_i,g_j)=\delta_i^j$ (the Kronecker symbol).
Set
\begin{eqnarray}
  \label{eq:555}
  \KoE:=2\sum_i\pi(g_i)\circ\pi(f_i)\in\End(\wedge^\bullet\lr^*).
\end{eqnarray}
Kostant defined a third operator
\begin{eqnarray}
  \label{eq:436}
  \KoR:=-\Li_0\circ\KoE \in\End(\wedge^\bullet\lr^*),
\end{eqnarray}
he proved that $\KoR$ is nilpotent and he defined 
\begin{eqnarray}
  \label{eq:205}
  s_w=(\Id-\KoR)^{-1}(h_w)=h_w+\KoR(h_w)+\KoR^2(h_w)+\cdots.
\end{eqnarray}
Here, we  need the  following improvement  of \cite[Lemma
4.6]{Kostant:harmform2} that proves the nilpotency of $\KoR$.

\begin{lemma}\label{lem:Rfilt}
For any integer $p$ and $\alpha\in X(Z)$, we have 
$$\KoR(C_{(\alpha,p)})\subset\bigoplus_{\beta\slleq\alpha}C_{(\beta,p)}.$$  
\end{lemma}

\begin{proof}
  Lemma~\ref{lem:opLstab} asserts that $\KoL$ stabilizes the
  $(X(Z)\times\ZZ)$-decomposition of $C$.
Since this decomposition is $\{\cdot,\cdot\}_\lr$-orthogonal by
Lemma~\ref{lem:decomporth}, this implies that $\KoL_0$ also stabilizes the
  $\tilde\Gamma$-decomposition of $C$.
By the formula~\eqref{eq:436}, it remains to prove that
$\KoE(C_{\alpha,p)})\subset\bigoplus_{\beta\slleq\alpha}C_{(\beta,p)}.$

But each $\pi(f_i)$ vanishes on $\wedge^\bullet\lu^*$ and each
$\pi(g_i)$ respects the degree. It follows that
$\KoE(C_{(\alpha,p)})\subset\bigoplus_{\beta\in X(Z)}C_{(\beta,p)}.$
But $\pi(g_i)$ vanishes on $\wedge^\bullet(\lu^-)^*$.
Moreover, $f_i$ belongs to $\lu^-$ and has a weight $\gamma\lleq 0$.
It follows that $\pi(f_i)(\wedge^\bullet(\lu^-)^*_{\beta})\subset
\wedge^\bullet (\lu^-)^*_{\beta-\gamma}$.
The claim follows.
\end{proof}

\subsubsection{Kostant's theorem}

\begin{theo}(\cite{Kostant:harmform2} )
\label{th:Kostant}
Let $w\in W^P$.
  The element $s_{w}\in\bigwedge^\bullet \lr^*$ defined by
  \eqref{eq:205} is $L$-invariant. 
In particular, $ s_{w}$
  corresponds by the
  isomorphism~\eqref{eq:139}
to a $K$-invariant  form $\omega_{w}$ on $G/P$.

Then the form $\omega_w$ is closed and its class $[\omega_w]$ in
$\Ho^*_{DR}(G/P,\CC)$ is equal to the  Schubert class $\sigma_w^\vee$, up to a
positive real scalar.
\end{theo}

\subsubsection{Application}
\label{sec:proofs}

We can now prove Lemma~\ref{lem:Schubfil}. 

\bigskip
\begin{proof}[of  Lemma~\ref{lem:Schubfil}] 
By Theorem~\ref{th:Kostant}, it is sufficient to prove that $\omega_w$
belongs to $F^{\lleq \tilde\rho(w)}\AAf^{l(w)}(G/P,\CC)$.
But $\omega_w$ and the filtration are $K$-invariant on the
$K$-homogeneous space $G/P$.
Hence it is sufficient to prove that $s_w$ belongs to $F^{\lleq \tilde\rho(w)}C$.
This is a consequence of  the property~\eqref{eq:144} and  Lemma~\ref{lem:Rfilt}.
\end{proof}

\bigskip
\begin{proof}[of  Proposition~\ref{prop:filtSchub}] 
Let $\tilde\beta\in \tilde\Gamma$ and let $p$ be an integer such that
$0\leq p\leq\dim(G/P)$.
 Consider $F^{\lleq\tilde\beta}\Ho^{2p}(G/P,\CC)$.
On one hand,
 Lemma~\ref{lem:dimF} shows that the dimension of
 $F^{\lleq\tilde\beta}\Ho^p(G/P,\CC)$ is not more than the cardinality
 of the set 
$$
W(\tilde\beta,p)=\{w\in W^P\,:\, \tilde\rho(G/P)-\tilde\rho(X_w) \lleq \tilde\beta\mbox{ and }l(w)=n-p\}.
$$
On the other hand, Lemma~\ref{lem:Schubfil} shows that
$F^{\lleq\tilde\beta}\Ho^p(G/P,\CC)$ contains the classes $\sigma_{w^\vee}$ for
$w$ in the set
$$
W'(\tilde\beta,p)=\{w\in W^P\,:\, \tilde\rho(X_w)\lleq \tilde\beta\mbox{ and }l(w)=p\}.
$$
But Lemma~\ref{lem:calcul} implies that the Poincar\'e duality
$w\mapsto w^\vee$ induces a bijection between 
$W(\tilde\beta,p)$ and $W'(\tilde\beta,p)$.
Since the family $(\sigma_{w^\vee})_{w\in W'(\tilde\beta,p)}$ is
linearly independant the proposition follows.
\end{proof}

\bigskip
\begin{proof}[of Corollary~\ref{cor:PoincareGP}]
The corollary is a direct consequence of Lemma~\ref{lem:PDdim} and the
above proof of Proposition~\ref{prop:filtSchub}.
\end{proof}

\subsection{Proof of Theorem~\ref{th:iso}}
\label{sec:proofth}

Let $u$ and $v$ be elements of $W^P$. Consider the following product in the
ordinary cohomology ring $\Ho^*(G/P,\CC)$
$$
\sigma_u.\sigma_v=\sum_{w\in W^P}c_{u\,v}^w\sigma_w.
$$
By Lemma~\ref{lem:Schubfil} and Lemma~\ref{lem:calcul}, $\sigma_u$ 
belongs to $F^{\lleq
  \tilde\rho(G/P)-\tilde\rho(X_u)}\Ho^{l(w_0w_o^P)-l(u)}(G/P,\CC)$.
Similarly, $\sigma_v$ belongs to
 $F^{\lleq
  \tilde\rho(G/P)-\tilde\rho(X_v)}\Ho^{l(w_0w_o^P)-l(v)}(G/P,\CC)$.
Now Proposition~\ref{prop:filcohom} shows that
$$
\sigma_u.\sigma_v\in F^{\lleq
  2\tilde\rho(G/P)-\tilde\rho(X_u)-\tilde\rho(X_v)}\Ho^{2l(w_0w_o^P)-l(u)-l(v)}(G/P,\CC).
$$
By Proposition~\ref{prop:filtSchub}, this means that
\begin{eqnarray}
  \label{eq:786}
  c_{u\,v}^w\neq 0&\Rightarrow&
\tilde\rho(G/P)-\tilde\rho(X_w)\,\lleq\,
2\tilde\rho(G/P)-\tilde\rho(X_u)-\tilde\rho(X_v),\\
&\Rightarrow&\tilde\rho(X_u)+\tilde\rho(X_v)\,\lleq\,\tilde\rho(X_w)+\tilde\rho(G/P).
\end{eqnarray}
Proposition~\ref{prop:filtSchub} implies also that
\begin{eqnarray}
  \label{eq:787}
  \overline{\sigma_u}\,.\, \overline{\sigma_v}=\sum_{\substack{w\in W^P\\
\tilde\rho(X_u)+\tilde\rho(X_v)=\tilde\rho(X_w)+\tilde\rho(G/P)}}
c_{u\,v}^w \overline{\sigma_w}.
\end{eqnarray}
On the other hand, Proposition~\ref{Prop_Dim} shows that
\begin{eqnarray}
  \label{eq:788}
  \sigma_u\bkprod \sigma_v=\sum_{\substack{
      w\in W^P\\
gd(X_u)+gd(X_v)=gd(X_w)+gd(G/P)}
}c_{u\,v}^w \sigma_w.
\end{eqnarray}
Comparing the identities~\eqref{eq:787} and \eqref{eq:788}, it remains to
prove, under the assumption $c_{u\,v}^w\neq 0$, that the equivalence
$$
\tilde\rho(X_u)+\tilde\rho(X_v)=\tilde\rho(X_w)+\tilde\rho(G/P)
\iff
gd(X_u)+gd(X_v)=gd(X_w)+gd(G/P)
$$
holds.

The implication ``$\Leftarrow$'' is an immediate consequence of the
definition~\eqref{eq:30} of $\rho(\cdot)$.
Conversely,  assume that
$\tilde\rho(X_u)+\tilde\rho(X_v)=\tilde\rho(X_w)+\tilde\rho(G/P)$.
Since $c_{u\,v}^w\neq 0$,  the Belkale-Kumar numerical criterion of
Levi-movability (see \cite[Theorem~15]{BK}) implies that
$\sigma_u\bkprod \sigma_v\bkprod\sigma_{w^\vee}=c_{u\,v}^w[pt]$.
In particular, Proposition~\ref{Prop_Dim} implies that
$gd(X_u)+gd(X_v)=gd(X_w)+gd(G/P)$.
The theorem is proved.

\subsection{The Belkale-Kumar fundamental class}

Recall from Section~\ref{sec:defclassY} the definition of the
Belkale-Kumar fundamental class of any subvariety of $G/P$.
We can now give  a simple characterization of this class using the
notion of $X(Z)$-dimension.

\begin{prop}
  \label{prop:BKclass}
Let $Y$ be an irreducible subvariety of $G/P$ of dimension $d$.
Consider the expansion of its  fundamental class in the Schubert basis
$$[Y]=\sum_{w\in W^P}d_w\sigma_w.$$
Then the expansion of its  $\bkprod$-fundamental class in the Schubert basis
is
$$[Y]_\bkprod=\sum_{\substack{  
w\in W^P\\
\rho(X_w)=\rho(Y)}}
d_w\sigma_w^\bkprod.$$
\end{prop}

\begin{proof}
It remains to prove that for any $[\omega]\in
\G^{\tilde\rho(Y)}\Ho^{2d}(G/P,\CC)$, 
$$
\int_Y\omega=[\omega]_{\bkprod}\bkprod(\sum_{\substack{
    w\in W^P\\
\rho(X_w)=\rho(Y)}}
d_w\sigma_w).
$$
Since the two members of the equality depend linearly on $[\omega]$, it
is sufficient to prove it for $[\omega]=\sigma_{u^\vee}$, for any
$u\in W^P$ such that $\rho(X_u)=\rho(Y)$ and $l(u)=d$.
By ordinary Poincar\'e duality, this case is equivalent to the following
equality
$$
\sigma_{u^\vee}.(\sum_{\substack{
    w\in W^P\\
l(w)=n-d}}
d_w\sigma_w)=\sigma_{u^\vee}\bkprod(\sum_{\substack{
    w\in W^P\\
\rho(X_w)=\rho(Y)}}
d_w\sigma_w).
$$
Since the only product $\sigma_{u^\vee}.\sigma_w$ that is nonzero in
the above formula is $\sigma_{u^\vee}.\sigma_u$, the proposition follows.
\end{proof}

\section{Intersecting Schubert varieties}

Given  $u,v\in W^P$ such that $v^\vee\bo u$, we construct in this
section a familly of varieties containing both the Richardson variety
$X_u\cap w_0 X_v$ (up to translation) and the variety $\Sigma_u^v$.
We prove (see Proposition~\ref{prop:rec}) that Conjecture~\ref{conj}
holds for $\Sigma_u^v$ if and only if it holds for all these varieties.
To end this section, we show that Conjecture~\ref{conj} is equivalent
to a formula using the Kostant harmonic forms that looks like a Fubini
formula.
 
\subsection{Products on $\Ho^*(G/P,\CC)$ and Bruhat orders}

The Bruhat order on $W^P$ is defined by
$$
u\bo v\ \iff\ X_u\subset X_v.
$$
This order is generated by $u<v$ if $l(v)=l(u)+1$ and $v=s_\alpha u$ for some positive root $\alpha$.
The weak Bruhat order on $W^P$ is  generated by the
relation
$
u\wbo v
$ if $l(v)=l(u)+1$ and $v=s_\alpha u$ for some simple root $\alpha$.
The relation between these two orders is
\begin{eqnarray}
  \label{eq:437}
u\wbo v\ \Rightarrow u\bo v.
\end{eqnarray}

A useful characterization of the weak Bruhat order is given by the
following result (see \cite{Bourb}).

\begin{lemma}
  \label{lem:wboPhi}
Let $u$ and $v$ in $W^P$.
Then $u\wbo v$ if and only if $\Phi(u)$ is contained in $\Phi(v)$.
\end{lemma}

The following relation  between the cup product and the Bruhat
order is well known
$$
\sigma_u.\sigma_v\neq 0\ \iff\ v^\vee\bo u.
$$
We have the following relation between the Belkale-Kumar product and the weak Bruhat
order.
\begin{lemma}
  \label{lem:BKprodwbo}Let $u$ and $v$ in $W^P$.
If 
$\sigma_u\bkprod\sigma_v\neq 0$  then $v^\vee\wbo u$.
\end{lemma}

\begin{proof}
  By assumption, there exists  $w\in W^P$ such that  $(u,v,w)$ is
  Levi-movable and $l(u)+l(v)+l(w)=l(w_0w_0^P)$.
Hence, for $(l_1,\,l_2,\,l_3)$ in a nonempty open subset of $L^3$:
$$
l_1T_u\cap l_2T_v\cap l_3 T_w=\{0\}.
$$
In particular, $l_1T_u+ l_2T_v=T_{P/P}G/P$.
Since $\Delta L.(B,w_0^PB_L)$ is open in $L^2$, there exist $l\in L$,
$b_1,b_2\in B_L$ such that
  $lb_1T_u+ lw_0^Pb_2T_v=T_{P/P}G/P$.
But $T_u$ and $T_v$ are $B_L$-stable and $T_{P/P}G/P$ is $L$-stable, hence   
$$
T_u+ w_0^PT_v=T_{P/P}G/P.
$$
It follows that $\Phi(u)\cup w_0^P\Phi(v)=\Phi(G/P)$.
But
$ \Phi({v^\vee})=\Phi(G/P)-w_0^P\Phi(v)$. Hence
$\Phi({v^\vee})\subset\Phi(u)$
and  $v^\vee\wbo u$.
\end{proof}

\bigskip
\begin{remark}
  The converse of the implication of Lemma~\ref{lem:BKprodwbo} does
  not hold.
Indeed consider $\SL_3 (\CC)$ with its usual maximal torus and Borel
subgroup $B$. Denote the two simple reflections of $W$ by $s_1$ and $s_2$.  
Then $\sigma_{s_1s_2}\bkprod\sigma_{s_2s_1}=0$ while $(s_2s_1)^\vee=s_2\wbo s_1s_2$.
\end{remark}

\subsection{Like Richardson's varieties}

Let $u,v\in W^P$. The Richardson variety $X_u^v$ is defined by
$$
X_u^v=X_u\cap w_0 X_v.
$$ 
It is well known that $X_u^v$ is irreducible, normal and satisfies
$[X_u^v]=\sigma_u.\sigma_v$.
In particular, $X_u^v$ is empty if and only if $v^\vee\bo u$.

Assume now that $v^\vee\wbo u$.
Fix $y\in W^P$ such that  $v^\vee\wbo y\wbo u$.
Consider the intersection
\begin{eqnarray}
  \label{eq:5}
  I_u^v(y):=y^{-1}X_u\cap w_0^Pv^{-1}.X_v.
\end{eqnarray}
The first example  $I_u^v(v^\vee)=(v^\vee)^{-1}X_u^v$ is just a
translated  Richardson variety.
 
By the relation~\eqref{eq:437}, the point $yP/P$ belongs to $X_u$. It
follows that $P/P$ belongs to $y^{-1}X_u$. Since $vP/P$ belongs to
$X_v$, $P/P$ belongs to $w_0^Pv^{-1}.X_v$. It follows that
\begin{eqnarray}
  \label{eq:436b}
P/P\in   I_u^v(y).
\end{eqnarray}

The following lemma shows that the variety $ I_u^v(y)$ contains a
translated Richardson variety.

\begin{lemma}
  \label{lem:RiinclusI}
Let $u$, $v$, and $y$ in $W^P$ such that $v^\vee\wbo y\wbo u$. Then 
$I_u^{y^\vee}(y)$ is contained in $I_u^v(y)$.
\end{lemma}

\begin{proof}
It remains to prove that 
$y^{-1}X_u\cap w_0^P(y^\vee)^{-1}.X_{y^\vee}$ is contained in
$y^{-1}X_u\cap w_0^Pv^{-1}.X_v$.
It is sufficient to prove that
$(y^\vee)^{-1}.X_{y^\vee}$ is contained in $v^{-1}.X_v$.
But $(y^\vee)^{-1}.X_{y^\vee}=\overline{((y^\vee)^{-1}By^\vee).P/P}$
and $v^{-1}.X_v =\overline{(v^{-1}Bv).P/P}$.
Hence it is sufficient to prove that $\Phi(\lieg/\lp,T)\cap
(y^\vee)^{-1}\Phi^+$
is contained in $\Phi(\lieg/\lp,T)\cap
v^{-1}\Phi^+$. 
But $v^\vee\wbo y$ and hence $y^\vee\wbo v$. 
Lemma~\ref{lem:wboPhi} allows to conclude.
\end{proof}

\bigskip
The fact that $X_u$ and $X_v$ are $B$-stable implies that the group
$H_u^v(y):=y^{-1}By\cap w_0^Pv^{-1}Bvw_0^P$ acts on $I_u^v(y)$.
Set $y'=y(v^\vee)^{-1}$ in such a way that $y=y'v^\vee$.
Note that
 $ yw_0^Pv^{-1}=y'w_0 $
and that
\begin{eqnarray}
  \label{eq:6v}
H_u^v(y)=(v^\vee)^{-1}(y'^{-1}By'\cap   B^-)v^\vee.
\end{eqnarray}
The group $H_u^v(y)$  is a connected subgroup of $G$,
containing  $T$
 and  acting on $I_u^v(y)$. 
Consider now the group $U(y')= y'^{-1}Uy'\cap   U^-$.

Let $\overset{\circ}{G/P}=B^-P/P$ denote the open $T$-stable affine cell containing
$P/P$. Set $\overset{\circ}I{}_u^v(y)=\overset{\circ}{G/P}\cap
I_u^v(y)$;
it is an open $T$-stable affine neighborhood of $P/P$ in $I_u^v(y)$.  
The following statement describes the geometry of this neighborhood.

\begin{theo}
\label{th:structloc}
Let $u$, $v$, and $y$ in $W^P$ such that $v^\vee\wbo y\wbo u$.
Then the following morphism
$$
\begin{array}{cccc}
  \Psi\;:&U(y')\times \overset{\circ}I_u^{y^\vee}{}(y)&\longto&\overset{\circ}I{}_u^v(y)\\
&(u,x)&\longmapsto&(v^\vee)^{-1} u v^\vee.x
\end{array}
$$
is an isomorphism.
\end{theo}

\begin{proof}
The weights of $T$ acting on the Lie algebra of the group
$U(y)=U^-\cap y^{-1}Uy$ are $\Phi(y)=\Phi^-\cap y^{-1}\Phi^+$.
The weights of $T$ acting on the tangent space at the point $P/P$
of the variety $w_0^P(y^\vee)^{-1}X_{y^\vee}$ are $\Phi(\lieg/\lp,T)\cap
y^{-1}\Phi^-$.
But   $\overset{\circ}{G/P}$ is isomorphic as a $T$-variety to the affine
  space $\lieg/\lp$.
It follows that the map 
\begin{eqnarray}
  \label{eq:829}
\begin{array}{ccc}
  U(y)\times [w_0^P(y^\vee)^{-1}X_{y^\vee}\cap
  \overset{\circ}{G/P}]&\longto&\overset{\circ}{G/P}\\
(u,x)&\longmapsto&ux
\end{array}
\end{eqnarray}
is an isomorphism.
The variety $y^{-1}X_u$ is stable by $y^{-1}By$ and so by $U(y)$.
It follows that  the map 
$$
\begin{array}{ccc}
  U(y)\times [w_0^P(y^\vee)^{-1}X_{y^\vee}\cap
  \overset{\circ}{G/P}\cap y^{-1}X_u]&\longto&\overset{\circ}{G/P}\cap
  y^{-1}X_u\\
(u,x)&\longmapsto&ux
\end{array}
$$
is an isomorphism.

Since $v^\vee\wbo y$ and $y=y'v^\vee$, the set $\Phi(y)$ is the
disjoint union of $\Phi(v^\vee)$ and $(v^\vee).\Phi(y')$ (see for
example \cite{Bou}).
Then the map
$$
\begin{array}{ccc}
  U(y')\times U(v^\vee)&\longto&U(y)\\
(u',u)&\longmapsto&(v^\vee)^{-1}u'v^\vee u
\end{array}
$$
is an isomorphism.
Note that in the above expression we have fixed
representative (still denoted by $v^\vee$) of $v^\vee$ in the normalizer of the torus $T$.
Composing these isomorphisms gives the following one:
$$
\begin{array}{ccc}
  U(y')\times U(v^\vee)\times 
  \overset{\circ}I{}_u^{y^\vee}(y)&\longto&\overset{\circ}{G/P}\cap
  y^{-1}X_u\\
(u',u,x)&\longmapsto&(v^\vee)^{-1}u'v^\vee ux.
\end{array}
$$
Since $\Phi(y')$ is contained in $(v^\vee)^{-1}\Phi^-$, and
$w_0^Pv^{-1}X_v=\overline{((v^\vee)^{-1}B^-v^\vee).P/P}$, the variety 
$w_0^Pv^{-1}X_v$ is stable under the action of $U(y')$.
Hence 
$$
\begin{array}{ccc}
  U(y')\times \left[ (U(v^\vee)\cdot 
  \overset{\circ}{I}{}_u^{y^\vee}(y))\cap w_0^Pv^{-1}X_v \right]&\longto&\overset{\circ}I{}_u^v(y)\\
(u',x)&\longmapsto&(v^\vee)^{-1}u'v^\vee x
\end{array}
$$
is an isomorphism.
It remains to prove that 
$$(U(v^\vee)\cdot 
  \overset{\circ}{I}{}_u^{y^\vee}(y))\cap w_0^Pv^{-1}X_v=
\overset{\circ}I{}_u^{y^\vee}(y).
$$
Let $u\in U(v^\vee)$ and $x\in \overset{\circ}{I}{}_u^{y^\vee}(y)$
such that $ux$ belongs to
$w_0^Pv^{-1}X_v$. It is sufficient to prove that $u=e$.
Replacing $y^\vee$ by $v$ in the morphism~\eqref{eq:829}, we get an
isomorphism
$$
\begin{array}{cccc}
  \Theta\;:&U(v^\vee)\times [w_0^Pv^{-1}X_{v}\cap
  \overset{\circ}{G/P}]&\longto&\overset{\circ}{G/P}\\
&(u',x')&\longmapsto&u'x'.
\end{array}
$$
One can easily 
check that $x$ belongs to $w_0^Pv^{-1}X_{v}\cap
  \overset{\circ}{G/P}$ and that $\Theta(u,x)=\Theta(e,ux)$.
Now, the
  injectivity of $\Theta$ implies that $u=e$.
\end{proof}

\bigskip
An important consequence of Theorem~\ref{th:structloc} for our purpose
is the following statement.

\begin{coro}
\label{lem:Ilocirred}
  The variety $I_u^v(y)$ is normal  at the point $P/P$. 
In particular, there exists an unique irreducible component
$\Sigma_u^v(y)$  of $I_u^v(y)$ which contains $P/P$.
\end{coro}

\begin{proof}
  The corollary follows from the theorem and the fact that the
  Richardson varieties are irreducible and normal (see \cite{KWY:Rich} for a
  short proof).
\end{proof}

\bigskip
If $y=v^\vee$ then Theorem~\ref{th:structloc} is trivial. In the
opposite situation when $y=u$ it implies the following result.

\begin{coro}
\label{cor:quasihom}
Let $u$ and $v$ in $W^P$ such that $v^\vee \wbo u$.
  The orbit $H_u^v(u).P/P$ is open in $I_u^v(u)$.
In other words, $\Sigma_u^v(u)$ is the closure of $H_u^v(u).P/P$.
\end{coro}

\begin{proof}
If $y=u$ then the translated Richardson variety $I_u^{y^\vee}(y)=I_u^{u^\vee}(u)$ is
reduced to the point $P/P$. The corollary follows immediately. 
\end{proof}

\subsection{A conjecture}

Here comes our main conjecture.

\begin{conj}
\label{conj}  
Let $u,v\in W^P$ such that $v^\vee\wbo u$. 
Then
$$
[\Sigma_u^v(u)]_\bkprod=\sigma_u^\bkprod\bkprod\sigma_v^\bkprod.
$$
\end{conj}

Some observations on this conjecture are collected in the following
propositions.

\begin{prop}\label{prop:zero}
  Expand  $[\Sigma_u^v(u)]_\bkprod$ and $\sigma_u\bkprod\sigma_v$ in
  the Schubert basis:
$$
\begin{array}{r@{\:}c@{\:}l}
[\Sigma_u^v(u)]_\bkprod&=&\sum_{w\in
    W^P}d_{uv}^w\sigma_w^\bkprod, \mbox{ and}\\[1em]
\sigma_u^\bkprod\bkprod\sigma_v^\bkprod&=&\sum_{w\in
    W^P}\tilde c_{uv}^w\sigma_w^\bkprod.
\end{array}$$
Then, for any $w\in W^P$,
\begin{enumerate}
\item $\tilde c_{uv}^w\geq d_{uv}^w$;
\item $\tilde c_{uv}^w\neq 0\;\iff\;d_{uv}^w\neq 0$.
\end{enumerate}
\end{prop}

\begin{proof}
  Write $[\Sigma_u^v(u)]=\sum_{w\in
    W^P}e_{uv}^w\sigma_w$
and 
$\sigma_u.\sigma_v=\sum_{w\in
    W^P} c_{uv}^w\sigma_w$
 in ordinary cohomology.
Since $\Sigma_u^v(u)$ is an irreducible component of the intersection
$I_u^v(u)$ and that this intersection is proper along this component,
the inequality
\begin{eqnarray}
  \label{eq:8814}
  c_{uv}^w\geq e_{uv}^w
\end{eqnarray}
holds for any $w\in W^P$.
Consider now a coefficient $d_{uv}^w$ for some fixed $w\in W^P$.
If $d_{uv}^w=0$ then the first assertion of the proposition is
obvious. Assume  $d_{uv}^w\neq 0$.
By Proposition~\ref{prop:BKclass}, $d_{uv}^w=e_{uv}^w$.
Comparing the inequality~\eqref{eq:8814} and the first assertion, one
observes that it is sufficient to prove that $\tilde
c_{uv}^w=c_{uv}^w$; that is that $\tilde c_{uv}^w\neq 0$.

Since $d_{uv}^w\neq 0$, Proposition~\ref{prop:BKclass} implies that
$\rho(X_w)=\rho(\Sigma_u^v(u))$.
Since $P/P$ belongs to the open $H_u^v(u)$-orbit in $\Sigma_u^v(u)$
and
is $X(Z)$-regular. In particular 
$$
\rho(\Sigma_u^v(u))=\sum_{\gamma\in X(Z)}\dim[(T_{P/P}\Sigma_u^v(u))_\gamma]
\gamma,
$$
where $(T_{P/P}\Sigma_u^v(u))_\gamma$ is the weight space of weight
$\gamma$ of the $Z$-module $T_{P/P}\Sigma_u^v(u)$.
But $T_{P/P}\Sigma_u^v(u)$ is the transverse intersection of
$T_{P/P}u^{-1}X_u$ and $T_{P/P}w_0^Pv^{-1}X_v$. It follows that 
$\rho(\Sigma_u^v(u))=\rho(u^{-1}X_u)+\rho(w_0^Pv^{-1}X_v)=
 \rho(X_u)+\rho(X_v)$.
 Finally  $\rho(X_w)=\rho(X_u)+\rho(X_v)$
and Proposition~\ref{prop:BKclass} shows that $\tilde c_{uv}^w=c_{uv}^w$. \\

Assuming that $d_{uv}^w\neq 0$, the first assertion implies that
$\tilde c_{uv}^w\neq 0$. Assume conversely that $\tilde c_{uv}^w\neq
0$ in other words that $(u,v,w^\vee)$ is Levi-movable.
Arguing like in the proof of Lemma~\ref{lem:BKprodwbo}, one can check
that there exists $l\in L$ such that $u^{-1}X_u$,
$w_0^Pv^{-1}X_v$ and $l(w^\vee)^{-1}X_{w^\vee}$ intersect transversally
at $P/P$.
It follows immediately that $\Sigma_u^v(u)$ and 
 $l(w^\vee)^{-1}X_{w^\vee}$ intersect transversally at $P/P$.
Hence $e_{uv}^w\neq 0$.

It remains to prove that $e_{uv}^w=d_{uv}^w$.
The condition $\tilde c_{uv}^w\neq 0$ in the $X(Z)$-graded algebra
$\G \Ho^*(G/P,\CC)$ implies that $\rho(X_w)=\rho(X_u)+\rho(X_v)$.
But $\rho(X_u)+\rho(X_v)=\rho(\Sigma_u^v(u))$.
Proposition~\ref{prop:BKclass} shows that $e_{uv}^w=d_{uv}^w$.
\end{proof}

\bigskip
\begin{prop}
\label{prop:rec}
  Let $u,v\in W^P$ such that $v^\vee\wbo u$. 
  \begin{enumerate}
  \item Conjecture~\ref{conj} holds if $\Sigma_u^v(u)$ has dimension
    $0$,  $1$ or $2$.
\item Conjecture~\ref{conj} holds if and only if for any $y\in W^P$
  such that $v^\vee\wbo y\wbo u$ we have $
[\Sigma_u^v(y)]_\bkprod=\sigma_u\bkprod\sigma_v
$.
  \end{enumerate}
\end{prop}

\begin{proof}
  If $\Sigma_u^v(u)$ has dimension $0$ then $u=v^\vee$. In particular 
$[\Sigma_u^v(y)]_\bkprod=[pt]=\sigma_u\bkprod\sigma_v$.\\

 If $\Sigma_u^v(u)$ has dimension $1$ then $u=s_\alpha v^\vee$, for
 some simple root $\alpha$.
Moreover, $l(u)=l(v^\vee)+1$.
This implies that $X_u$ is stable by the action of $P_\alpha$ (the
minimal parabolic subgroup associated to $\alpha$).
In particular $s_\alpha X_u=X_u$.
It follows that
$u^{-1}X_u=u^{-1} s_\alpha X_u=(v^\vee)^{-1}X_u$. In particular
$I_u^v(u)=I_u^v(v^\vee)$ is a translated Richardson variety
and is irreducible.
Moreover, 
$\sigma_u.\sigma_v=[I_u^v(u)]=[\Sigma_u^v(u)]$.
Proposition~\ref{prop:BKclass} implies that $\sigma_u\bkprod\sigma_v
=[\Sigma_u^v(u)]_\bkprod$.\\

Assume now that $u=s_\alpha s_\beta v^\vee$, for
 some simple roots $\alpha$ and $\beta$ such that $l(u)=l(v^\vee)+2$.
Then (note that $s_\alpha X_u=X_u$)
$$
\begin{array}{ccl}
  I_u^v(u)&=&u^{-1}X_u\cap w_0^Pv^{-1}X_v\\
&=&(v^\vee)^{-1}(s_\beta s_\alpha X_u\cap w_0 X_v)\\
&=&(v^\vee)^{-1}s_\beta (s_\alpha X_u\cap w_0 s_{\beta^*}X_v),\\
\end{array}
$$
where $\beta^*=-w_0\beta$.
But the condition $v^\vee\wbo s_\beta v^\vee$ implies that 
$s_{\beta^*}v\wbo v$ (see for example Lemma~\ref{lem:wboPhi}).
Then $s_{\beta^*}X_v=X_v$ and 
$I_u^v(u)$ is obtained by translation from the Richardson variety
$s_\alpha X_u\cap w_0 s_{\beta^*}X_v$.
The first assertion of the proposition follows.\\

Let $\alpha$ be a simple root such that $y\wbo s_\alpha y\wbo u$. 
Set $\beta=-y\inv \alpha$ and set $U_\beta\,:\,\CC\longto G$, the associated additive one-parameter subgroup.
Consider the flat limit 
$\lim_{t\to\infty} U_\beta(t)\Sigma_u^v(y)$.
Since $U_\beta(t)y\inv B/B$ tends to $y\inv s_\alpha B/B$ when $t$
goes to infinity, 
 $\lim_{t\to\infty} U_\beta(t)y\inv X_u=y\inv s_\alpha X_u$. 
Since  $v^\vee\wbo y\wbo s_\alpha y$, $\beta\in \Phi(s_\alpha
y)-\Phi(v^\vee)$ and $w_0^P\beta\in\Phi(v)$. 
In particular, $w_0^Pv\inv X_v$ is $U_\beta$-stable. 
But $\Sigma_u^v(s_\alpha y)$ is an irreducible component of the
intersection $y\inv s_\alpha X_u\cap w_0^Pv\inv X_v$ ; and, this
intersection is transverse along this component.
It follows that $\Sigma_u^v(s_\alpha y)$ is an irreducible component
of $\lim_{t\to\infty} U_\beta(t)\Sigma_u^v(y)$.
Writing
$$
[\Sigma_u^v(y)]=\sum_{w\in W^P}d'_w\sigma_w
\quad{\rm and}\quad
[\Sigma_u^v(s_\alpha y)]=\sum_{w\in W^P}d''_w\sigma_w,
$$
we deduce that 
\begin{equation}
  \label{eq:51}
  d''_w\leq d'_w\qquad\forall w\in W^P.
\end{equation}
Write now
$$
[\Sigma_u^v(v^\vee)]=\sum_{w\in W^P}d_w\sigma_w
\quad{\rm and}\quad
[\Sigma_u^v(u)]=\sum_{w\in W^P}e_w\sigma_w.
$$
Since $\Sigma_u^v(v^\vee)$ is a translated of the Richardson variety
$X_u\cap w_0 X_v$, 
$$
\sigma_u.\sigma_v =\sum_{w\in W^P}d_w\sigma_w.
$$
By an immediate induction, we deduce from \eqref{eq:51} that 
$$
e_w\leq d'_w\leq d_w \qquad\forall w\in W^P.
$$
Conjecture~\eqref{conj} holds for $y=u$ if and only if for any $w\in W^P$
such that $(u,v,w)$ is Levi-movable $e_w=d_w$. 
Then, $d'_w=d_w$ for any such $w\in W^P$ and 
$
[\Sigma_u^v(y)]_\bkprod=\sigma_u\bkprod\sigma_v
$.
\end{proof}

\subsection{Interpretation in terms of harmonic forms}

Kostant's harmonic forms allow to formulate Conjecture~\ref{conj} as
an identity of integrals.

\begin{prop}
 \label{prop:conjform}
Let $u$ and $v$ in $W^P$ such that $v^\vee\wbo u$.
Then $\sigma_u\bkprod\sigma_v=[\Sigma_u^v(u)]_\bkprod$ if and only if
for any $w$ in $W^P$ such that $(u,v,w)$ is Levi-movable, we have
$$
\int_{(u^\vee)^{-1}X_{u^\vee}}\omega_{u^\vee}.
\int_{(v^\vee)^{-1}X_{v^\vee}}\omega_{v^\vee}. \int_{\Sigma_u^v(u)}\omega_{w^\vee}=
\int_{(P^u)^-}\omega_{u^\vee}\wedge\omega_{v^\vee}\wedge \omega_{w^\vee}.
$$ 
\end{prop}

\begin{proof}
For any $w\in W^P$, consider the Kostant's harmonic form $\omega_w$ and the nonzero
complex number $\lambda_w$ (see Theorem~\ref{th:Kostant}) such that 
\begin{eqnarray}
  \label{eq:914}
  [\omega_w]=\lambda_w\sigma_w^\vee.
\end{eqnarray}
Then
\begin{eqnarray}
  \label{eq:916}
  \lambda_w=\int_{w^{-1}X_w}\omega_w.
\end{eqnarray}
  By Propositions~\ref{prop:BKclass} and \ref{prop:zero},
  Conjecture~\ref{conj} is equivalent to the fact that for any $w\in
  W^P$ such that $(u,v,w)$ is Levi-movable, we have
  \begin{eqnarray}
    \label{eq:915}
    \sigma_u.\sigma_v.\sigma_w=[\Sigma_u^v(u)].\sigma_w.
  \end{eqnarray}
But on one hand
\begin{eqnarray}
  \label{eq:917}
  \sigma_u.\sigma_v.\sigma_w=
\frac{1}{\lambda_{u^\vee}\lambda_{v^\vee}}[\omega_{u^\vee}\wedge\omega_{v^\vee}].\sigma_w=
\frac{\int_{w^{-1}X(w)}\omega_{u^\vee}\wedge\omega_{v^\vee}}{\lambda_{u^\vee}\lambda_{v^\vee}}.
\end{eqnarray}
And on the other hand 
\begin{eqnarray}
  \label{eq:918}
  [\Sigma_u^v(u)].\sigma_w=\frac{\int_{\Sigma_u^v(u)}\omega_{w^\vee}}{\lambda_{w^\vee}}.
\end{eqnarray}
In particular the equality~\eqref{eq:915} is equivalent to 
\begin{eqnarray}
  \label{eq:919}
  \lambda_{w^\vee}.\int_{w^{-1}X_w}\omega_{u^\vee}\wedge\omega_{v^\vee}=
\lambda_{u^\vee}.\lambda_{v^\vee}. \int_{\Sigma_u^v(u)}\omega_{w^\vee};
\end{eqnarray}
which is,  by~\eqref{eq:916}, equivalent to
\begin{eqnarray}
  \label{eq:920}
\lambda_{w^\vee}
.\int_{w^{-1}X_w}\omega_{u^\vee}\wedge\omega_{v^\vee}=
\int_{(u^\vee)^{-1}X_{u^\vee}}\omega_{u^\vee}.
\int_{(v^\vee)^{-1}X_{v^\vee}}\omega_{v^\vee}. \int_{\Sigma_u^v(u)}\omega_{w^\vee}.
\end{eqnarray}
We claim that
\begin{eqnarray}
  \label{eq:921}
\lambda_{w^\vee}
.\int_{w^{-1}X_w}\omega_{u^\vee}\wedge\omega_{v^\vee}=
\int_{(P^u)^-}\omega_{u^\vee}\wedge\omega_{v^\vee}\wedge \omega_{w^\vee}.
\end{eqnarray}

Let $d$ be the positive integer such that
$\sigma_u.\sigma_v.\sigma_w=d[pt]$.
We have
$$d=\int_{G/P}\frac{\omega_{u^\vee}\wedge\omega_{v^\vee}\wedge\omega_{w^\vee}}{\lambda_{u^\vee}
  \lambda_{v^\vee} \lambda_{w^\vee}}.$$
Since $\sigma_u.\sigma_v=d\sigma_{w^\vee}$, we also have
$$d=\int_{w^{-1}X(w)}\frac{\omega_{u^\vee}\wedge\omega_{v^\vee}}{\lambda_{u^\vee}
  \lambda_{v^\vee}}.$$
Claim~\eqref{eq:921} is obtained by identifying these two expressions of
$d$.

The proposition follows now from the equations~\eqref{eq:921} and \eqref{eq:920}.
\end{proof}

\bigskip
\begin{remark}
  Observe that $(P^u)^-$ is isomorphic to the product of the three
  $T$-stable affine neighborhoods of $P/P$ in
  $(u^\vee)^{-1}X_{u^\vee}$, $(v^\vee)^{-1}X_{v^\vee}$ and
  $\Sigma_u^v(u)$. With this observation the equality of
  Proposition~\ref{prop:conjform} looks like a Fubini formula.
\end{remark}

\section{The case of the complete flag varieties}
\label{sec:GB}

Given $u$ in $W$, set $\Phi(u)^c:=\Phi^--\Phi(u)$.
Let $u,v$, and $w$ in $W$.
For the complete flag variety $G/B$ the Levi-movability is easy to
understand.
Indeed $T_u$, $T_v$, and $T_w$ are $L=T$-stable. 
In particular, 
$(\sigma_u,\sigma_v,\sigma_w)$ is Levi-movable if and only if
the natural map
$T_{B/B}(G/B) \longto \frac{T_{B/B}(G/B)}{T_u}\oplus
\frac{T_{B/B}(G/B)}{T_v}\oplus \frac{T_{B/B}(G/B)}{T_w}$
is an isomorphism.
This is equivalent to the fact that $\Phi^-$ is the disjoint union of
$\Phi(u)^c$, $\Phi(v)^c$, and $\Phi(w)^c$.
Since   $\Phi(w)^c=\Phi(w^\vee)$, one gets the following equivalence
$$
\tilde c_{uv}^w\neq 0\iff \Phi(w)^c=\Phi(u)^c\ccup\Phi(v)^c.
$$
Conjecture~\ref{conj} generalizes a classical one on $G/B$.

\begin{prop}
  \label{prop:casGB}
Let $G$ be a semisimple group and consider the Belkale-Kumar
cohomology of $G/B$.
Let $u$ and $v$ belong to $W$.
Then $\sigma_u\bkprod\sigma_v=[\Sigma_u^v(u)]_\bkprod$ if and only if
$\sigma_u\bkprod\sigma_v$ is either equal to zero or to 
$\sigma_w$ for some $w\in W$.
\end{prop}

\begin{proof}
  Assume that $\sigma_u\bkprod\sigma_v=[\Sigma_u^v(u)]_\bkprod$.

\noindent
Case 1. Suppose there exists $w\in W$ such that $\Phi(w)=\Phi(H_u^v(u))$.\\
Then  (see for example Lemma~\ref{lem:Tvar}) 
$\Sigma_u^v(u)=w^{-1}X_w$; hence $[\Sigma_u^v(u)]_\bkprod=\sigma_w$.
In particular $\sigma_u\bkprod\sigma_v=\sigma_w$.

\noindent
Case 2. Suppose there exists no $w\in W$ such that
$\Phi(w)=\Phi(H_u^v(u))$.\\
Since $\Phi(H_u^v(u))=\Phi(u)\cap\Phi(v)$, there is no $w\in W$ such that 
$\Phi(w)^c=\Phi(u)^c\ccup\Phi(v)^c$. Hence there is no $w\in W$ such that 
$(\sigma_u,\sigma_v,\sigma_{w^\vee})$ is Levi-movable.
Then
$\sigma_u\bkprod\sigma_v=0$.
Moreover Proposition~\ref{prop:filtSchub} implies that
$\G^{\tilde\rho(G/P)-\tilde\rho(\Sigma_u^v(u))} \Ho^*(G/P,\CC)=\{0\}$.
In particular, $[\Sigma_u^v(u)]_\bkprod=0$.
\\

Assume now that  $\sigma_u\bkprod\sigma_v=\sigma_w$ for some $w\in
W$.\\
Since $\Phi(w)^c=\Phi(u)^c\ccup\Phi(v)^c$, Lemma~\ref{lem:Tvar} 
shows that $w^{-1}X_w=\Sigma_u^v(u)$.
Hence $\sigma_u\bkprod\sigma_v=[\Sigma_u^v(u)]_\bkprod$.\\

Assume finally that  $\sigma_u\bkprod\sigma_v=0$.\\
It remains to prove that $[\Sigma_u^v(u)]_\bkprod=0$.
Since $\Phi(H_u^v(u))=\Phi(u)\cap\Phi(v)$,
 $[\Sigma_u^v(u)]_\bkprod$ belongs to $\G^{\rho(X_u)+\rho(X_v)}\Ho^*(G/B,\CC)$.
If there is no $w$ in $W$ such that $\rho(X_w)=\rho(X_u)+\rho(X_v)$ then
Proposition~\ref{prop:filtSchub} shows that 
$\G^{\rho(X_u)+\rho(X_v)}\Ho^*(G/B,\CC)=\{0\}$.
In particular $[\Sigma_u^v(u)]_\bkprod=0$.
 Assume now that there exists $w$ in $W$ such that $\rho(X_w)=\rho(X_u)+\rho(X_v)$.
Then $[\Sigma_u^v]=d\sigma_w+\cdots$ for some integer $d$.
If $d=0$ there is nothing to prove.
If $d\neq 0$ then $\sigma_u.\sigma_v=e\sigma_w +\cdots$ for some
integer $e\geq d$.
The numerical criterium \cite[Theorem~15]{BK} shows that
$\sigma_u\bkprod\sigma_v=e\sigma_w$.
This contradicts the assumption $\sigma_u\bkprod\sigma_v=0$.
\end{proof}

\bigskip
Proposition~\ref{prop:casGB} shows that, for $G/B$,
Conjecture~\ref{conj} is equivalent to the following one.

\begin{conj}
  \label{conjGB}
Let $u,v,$ and $w$ in $W$ such that
$\Phi(w)^c=\Phi(u)^c\ccup\Phi(v)^c$.
Then $\sigma_u\bkprod\sigma_v=\sigma_w$ in $\Ho^*(G/B,\CC)$.
\end{conj}

Conjecture~\ref{conjGB} was stated by Dimitrov and Roth in
\cite{DR:prv1}. 
If $G=\SL_n(\CC)$ then Conjecture~\ref{conjGB} was proved by Richmond in
\cite{Richmond:recursion}.
If $G=\Sp_{2n}(\CC)$ then Conjecture~\ref{conjGB} was proved independently  in
\cite{Rich:mult} and \cite{multi}.
Dimitrov and Roth have a proof for each simple classical $G$, but it
is  not published.
Here we 
include a proof for the group $\SO_{2n+1}(\CC)$.

\begin{prop}\label{prop:GBSOimpair}
  Conjecture~\ref{conjGB} holds for the group $\SO_{2n+1}(\CC)$.
\end{prop}

\begin{proof}
  Let $V$ be a $(2n+1)$-dimensional complex vector space 
and let ${\mathcal B}=(x_1,\dots,x_{2n+1})$ be a basis of $V^*$. 
Let $G$ be the special orthogonal group associated to the quadratic form 
$
Q=x_{n+1}^2+\sum_{i=1}^nx_ix_{2n+2-i}.
$ 
Consider the maximal torus  $T=\{{\rm
  diag}(t_1,\dots,t_n,1,t_n^{-1},\dots,t_1^{-1})\,:\,t_i\in\CC^*\}$
of $G$.
Let $B$ be the Borel subgroup of $G$ consisting of upper triangular
matrices in the dual base of ${\mathcal B}$.
Consider $W$, $\Phi$, $\Phi^+$ associated to $T\subset B\subset G$.

Let $u$, $v$, and $w$ in $W$ such that
$\sigma_u\bkprod\sigma_v\bkprod\sigma_w=d[pt]$
for some positive integer $d$.
It remains to prove that $d=1$.
The Levi-movability  implies that 
$
\Phi^-=\Phi(u)^c\ccup
\Phi(v)^c\ccup\Phi(w)^c$.

Consider the linear group $\hat G=\GL(V)$. Let $\hat T$ denote the
subgroup of $\hat G$ consisting of 
diagonal matrices and let $\hat B$  denote the
subgroup of $\hat G$ consisting of 
upper triangular matrices in $\hat G$.
Consider $\hat W$, $\hat \Phi$, $\hat \Phi^+$ associated to $\hat T\subset \hat B\subset \hat G$.
Since $T$ is a regular torus in $\hat G$, the group $W$ identifies
with a subgroup of $\hat W$. In particular, $u,v,$ and $w$ belong to
$\hat W$. One can easily check that the similar property of $\Phi^-$
implies that
$
\hat \Phi^-=
\hat \Phi(u)^c\ccup\hat \Phi(v)^c\ccup\hat \Phi(w)^c$.
Consider now the three Schubert varieties $\hat X_u$, $\hat X_v$, and
$\hat X_w$ in $\hat G/\hat B$.
The fact that Conjecture~\ref{conjGB} holds for $\hat G$ implies that 
\begin{eqnarray}
  \label{eq:753}
  u^{-1}\hat X_u\cap v^{-1}\hat X_v\cap w^{-1}\hat X_w=\{\hat B/\hat B\}.
\end{eqnarray}
Consider now the inclusion $G/B\subset \hat G/\hat B$.
Then $X_u$ is contained in $\hat X_u$ (and similar inclusions hold for $v$
and $w$). In particular, the condition~\eqref{eq:753} implies that
\begin{eqnarray}
  \label{eq:754}
u^{-1} X_u\cap v^{-1}X_v\cap w^{-1} X_w=\{B/ B\}.
\end{eqnarray}
Moreover, the condition on $\Phi^-$ implies that the intersection in
\eqref{eq:754} is transverse.
It follows that $d=1$.
\end{proof}

\begin{prop}\label{prop:GBcasordi}
  Conjecture~\ref{conjGB} holds for the groups of type $F_4$ and $E_6$.
\end{prop}

\begin{proof}
For $w\in W$ set
$$
p(w)=\prod_{\alpha\in\Phi^+\cap w\Phi^+}(\rho,\alpha),
$$
where $(\cdot,\cdot)$ is a $W$-invariant scalar product and $\rho$ is
the half sum of the positive roots.
  Let $u,v,$ and $w$ in $W$ such that
$\Phi(w)^c=\Phi(u)^c\ccup\Phi(v)^c$.
By \cite[Corollary~44]{BK},
$$
\sigma_u\bkprod\sigma_v=\frac{p(u).p(v)}{p(w)}\sigma_w
$$ in $\Ho^*(G/B,\CC)$.
To prove the proposition, it is sufficient to check that
$p(w)=p(u).p(v)$. 
This is checked by a Sage program (see~\cite{MaPage}).
For example, in type $F_4$, if 
$$
\begin{array}{cl}
u^\vee= s_3s_2s_3s_2, \quad 
v^\vee=
s_1s_2s_3s_4s_2s_3s_1s_2s_3s_4 &{\rm and}\\ 
w^\vee=
s_1s_2s_3s_4s_2s_3s_1s_2s_3s_4s_3s_2s_3s_2 
\end{array}
$$
then
$$
p(u)= \frac 3 2 \quad p(v)= 113400 \quad  p(w)= 170100.
$$
 And, in type $E_6$, if 
$$
\begin{array}{cl}
u= s_6s_5s_4s_3s_2s_4s_5s_6s_5s_3 \quad v=
  s_4s_3s_2s_4s_5s_4s_3s_2s_4s_2 \\ w=
  s_6s_5s_4s_3s_2s_4s_5s_6s_5s_4s_3s_2s_4s_5s_4s_3s_2s_4s_3s_2
\end{array}
$$
then
$$
p(u)= 20160 \quad p(v)= 4320 \quad p(w)= 87091200.
$$
\end{proof}

\bibliographystyle{amsalpha}
\bibliography{bkRichardson}

\begin{center}
  -\hspace{1em}$\diamondsuit$\hspace{1em}-
\end{center}

\end{document}